\documentclass[a4paper]{article}
\usepackage[utf8]{inputenc}
\usepackage{amsthm}
\usepackage{amssymb}
\usepackage{amsmath}
\usepackage{epsfig}
\usepackage{colonequals}
\usepackage{enumerate}
\usepackage{hyperref}
\usepackage{subcaption}
\usepackage{tikz}

\usepackage{amsfonts}       
\usepackage{bbding}         
\usepackage{bm}             
\usepackage{graphicx}       
\usepackage{fancyvrb}       
\usepackage{dcolumn}        
\usepackage{booktabs}       
\usepackage{paralist}       
\usepackage{amssymb} 

\theoremstyle{plain}
\newtheorem{theorem}{Theorem}
\newtheorem{lemma}[theorem]{Lemma}
\newtheorem{claim}[theorem]{Claim}
\newtheorem{corollary}[theorem]{Corollary}
\newtheorem{observation}[theorem]{Observation}
\newtheorem{definition}[theorem]{Definition}

\renewcommand{\SS}{\mathcal{S}}
\newcommand{\BB}{\mathcal{B}}
\newcommand{\TT}{\mathcal{T}}
\newcommand{\CC}{\mathcal{C}}
\newcommand{\KK}{\mathcal{K}}

\begin{document}
\title{Characterization of $4$-critical triangle-free toroidal graphs\footnote{Supported by the Neuron Foundation for Support of Science under Neuron Impuls programme.}}
\author{%
     Zdeněk Dvořák\thanks{Computer Science Institute (CSI) of Charles University,
           Malostranské náměstí 25, 118 00 Prague, 
           Czech Republic. E-mail: \protect\href{mailto:rakdver@iuuk.mff.cuni.cz}{\protect\nolinkurl{rakdver@iuuk.mff.cuni.cz}}.
	   }
    \and
    Jakub Pekárek\thanks{Charles University,
           Malostranské náměstí 25, 118 00 Prague, 
           Czech Republic. E-mail: \protect\href{mailto:edalegos@gmail.com}{\protect\nolinkurl{edalegos@gmail.com}}.}
}
\date{\today}
\maketitle

\begin{abstract}
We give an exact characterization of 3-colorability of triangle-free graphs
drawn in the torus, in the form of 186 ``templates'' (graphs with certain
faces filled by arbitrary quadrangulations) such that a graph from this class
is not 3-colorable if and only if it contains a subgraph matching one of the templates.
As a consequence, we show every triangle-free graph drawn in the torus
with edge-width at least six is 3-colorable, a key property used in an efficient
3-colorability algorithm for triangle-free toroidal graphs.
\end{abstract}

\section{Introduction}

By a fundamental theorem of Gr\"{o}tzsch~\cite{grotzsch1959}, every planar triangle-free
graph is 3-colorable.  This is not true for any surface of non-zero genus.
The only other surface to date for which the 3-colorability of triangle-free graphs is
completely understood is the projective plane; Gimbel and Thomassen~\cite{gimbel}
proved that a projective-planar triangle-free graph is 3-colorable if and only
if it does not contain a non-bipartite quadrangulation as a subgraph.

Recall a graph $G$ is \emph{k-critical} if its chromatic number is exactly $k$ and
every proper subgraph of $G$ has a chromatic number at most $k-1$, and thus a graph is
$c$-colorable if and only if it does not contain a $(c+1)$-critical subgraph.
Consequently, to characterize $c$-colorability in a subgraph-closed class of graphs,
it suffices to describe the $(c+1)$-critical graphs belonging to the class.
For example, the aforementioned result on projective-planar graphs can be restated as follows.
\begin{theorem}[Gimbel and Thomassen~\cite{gimbel}]
Projective-planar triangle-free $4$-critical graphs are exactly the
non-bipartite quadrangulations of the projective plane without non-facial contractible
4-cycles.
\end{theorem}
Let us remark that the situation is much simpler for graphs of girth at least five:
All such graphs drawn in the projective plane, the torus, or the Klein bottle
are $3$-colorable~\cite{thom-torus,tw-klein}, and in general, there are only finitely many
4-critical graphs of girth at least five that can be drawn in any fixed surface~\cite{thomassen-surf}.

The general theory of $3$-colorability of triangle-free graphs drawn on surfaces was
developed in a series of papers by Dvořák, Král' and Thomas, who in particular showed that
triangle-free 4-critical graphs are almost quadrangulations in the following sense.
The \emph{census} $\CC(G)$ of a graph $G$ drawn in a surface is the multiset of lengths of the faces of $G$, excluding the faces of length exactly four.
\begin{theorem}[Dvořák, Král' and Thomas~\cite{trfree4}]\label{thm-gen}
For every surface $\Sigma$, there exists $c_\Sigma$ such that the following claim holds.
If $G$ is a $4$-critical triangle-free graph drawn in $\Sigma$ without a non-contractible 4-cycle, then
$$\sum \CC(G)\le c_\Sigma.$$
\end{theorem}
In combination with a detailed 3-colorability theory for such almost quadrangulations~\cite{trfree6},
Theorem~\ref{thm-gen} can be used to design a linear-time algorithm to decide 3-colorability
of triangle-free graphs drawn in a fixed surface~\cite{trfree7}.

Let us remark that the assumption the graph does not contain a non-contractible 4-cycle
in the statement of Theorem~\ref{thm-gen} is necessary in general; based on a construction
of Thomas and Walls~\cite{tw-klein}, for any surface $\Sigma$ of sufficiently large genus,
there exist $4$-critical triangle-free graphs drawn in $\Sigma$ with arbitrarily large census
(nevertheless, the structure of these graphs is well understood~\cite{cylgen-part3}).
However, for torus the assumption can be dropped.

\begin{theorem}[Dvořák and Pekárek~\cite{torirr}]\label{thm-centor}
Suppose $H$ is a triangle-free $4$-critical graph drawn in the torus. Then $H$ is $2$-connected, the drawing of $H$ is $2$-cell,
has representativity at least $2$, and its census is $\emptyset$, $\{5,5\}$, $\{5,5,5,5\}$, $\{5,5,6\}$, or $\{5,7\}$.
\end{theorem}

The main result of this paper is an exact description of $4$-critical triangle-free toroidal graphs,
providing a precise characterization of 3-colorability of triangle-free graphs drawn in the torus.
There are infinitely many such $4$-critical graphs, but they all can be obtained by quadrangulating
faces of one of 186 template graphs.  More precisely, a \emph{simple template} $T$ consists of a graph $G_T$ $2$-cell embedded in the torus
and a subset $Q_T$ of its even-length faces.  A graph drawn in the torus $G$ \emph{is represented} by the simple template $T$ if
a graph homeomorphic to $G$ is obtained from $G_T$ by filling each face in $Q_T$ by an arbitrary quadrangulation.
\begin{theorem}\label{thm-main}
There exists a set $\TT$ of 186 simple templates with the following properties:
\begin{itemize}
\item Every $4$-critical triangle-free graph drawn in the torus is represented by a simple template belonging to $\TT$.
\item No graph represented by a simple template belonging to $\TT$ is $3$-colorable.
\end{itemize}
\end{theorem}
Consequently, a triangle-free graph $G$ drawn in the torus is $3$-colorable if and only if no subgraph of $G$
is represented by a simple template belonging to $\TT$.  The set $\TT$ is presented in the Appendix; more usefully,
it can be found in a computer-readable format (together with the programs we used to generate it) at
\href{https://iuuk.mff.cuni.cz/~rakdver/torus/}{\url{https://iuuk.mff.cuni.cz/~rakdver/torus/}}.

Let $T_0\in \TT$ be the simple template where $G_{T_0}$ is the graph $I_4$ depicted in
Figure~\ref{fig-irreducibles} and $Q_{T_0}=\emptyset$.  Note that $I_4$ is a quadrangulation and
it is the only graph represented by this simple template.  Every other simple template in $\TT$
represents an infinite family of $4$-critical triangle-free toroidal graphs.
Moreover, for every $T\in\TT\setminus\{T_0\}$, the graph $G_T$ contains a non-contractible $4$-cycle,
while $I_4$ has edge-width five (recall the \emph{edge-width} of a graph drawn in a surface
is the length of its shortest non-contractible cycle).
\begin{corollary}\label{cor-ew}
Let $G$ be a triangle-free graph drawn in the torus.
\begin{itemize}
\item If $G$ has edge-width at least six, then $G$ is $3$-colorable.
\item If $G$ has edge-width five, then $G$ is $3$-colorable if and only if it does not contain $I_4$ as a subgraph.
\end{itemize}
\end{corollary}
Let us remark that in~\cite{cylflow}, we proved 3-colorability for graphs of edge-width at least 21,
and used this property to design a simple algorithm to decide 3-colorability of triangle-free
toroidal graphs.  The bound from Corollary~\ref{cor-ew} substantially improves the multiplicative
constants in the time complexity of this algorithm.

\subsection*{Proof outline}

Thomassen~\cite{thom-torus} proved that every graph embedded in the torus without
contractible $(\le\!4)$-cycles (but possibly with non-contractible triangles or
$4$-cycles) is $3$-colorable.  Moreover, every contractible 4-cycle in an embedded
4-critical graph is know to bound a face.  Consequently, every $4$-critical triangle-free
graph $G$ drawn in the torus has a $4$-face. A natural way of dealing with 4-faces
is to identify opposite vertices of one of them, effectively removing it from the
graph. This operation is of particular interest as while reducing size of the
graph it never produces graphs with lower chromatic number, and thus the resulting
graph contains a 4-critical subgraph $G'$.  We say $G'$ is a \emph{reduction} of $G$.  
To inductively prove Theorem~\ref{thm-main}, we need to establish the following claims
for a $4$-critical triangle-free graph $G$ drawn in the torus:
\begin{itemize}
\item If for every 4-face of $G$, the identification of opposite vertices
creates a triangle (we call such graphs \emph{irreducible}), then $G$ is represented
by an element of $\TT$.
\item If $G'$ is a triangle-free reduction of $G$ and $G'$ is represented by an element of $\TT$,
then $G$ is also represented by an element of $\TT$.
\end{itemize}
In~\cite{torirr}, we took care of the former step, proving that there are
exactly four such irreducible graphs up to homeomorphism.
For the latter step, we need to study the inverse to the reduction process.
This inverse is quite well understood---it has been used in a similar way in~\cite{trfree4,torirr},
and inverses to more general reduction operations have been used in a number of other
papers on embedded critical graphs.  Nevertheless, we need to develop the machinery
for mimicking the inverse to reduction on templates (Theorem~\ref{thm-aplif}).  Finally, we finish the proof
via a computer-assisted case analysis, showing that any application of the inverse to reduction
operation on an element of $\TT$ results in another element of $\TT$.  More precisely,
we actually use computer-assisted enumeration, starting from simple templates
corresponding to the four irreducible graphs and adding further templates obtained
by applying the inverse to reduction operation, until we obtain a set closed on this operation.

Let us note that it is not a priori clear that this enumeration process succeeds;
indeed, on more complicated surfaces, a more general kind of templates (with non-2-cell quadrangulated faces)
has to be considered~\cite{cylgen-part3}.  Also, the general bounds from~\cite{trfree4,cylgen-part3}
on the sizes of templates are large, giving no guarantee that the enumeration can be
accomplished in practice.  Fortunately, neither of these concerns materialized.

\subsection*{Algorithmic remarks}

Theorem~\ref{thm-main} gives a way to test whether a triangle-free toroidal graph $G$ is 3-colorable,
by checking whether it contains a subgraph represented by an element $T\in\TT$.  Performing this test
efficiently is not entirely trivial, though.  After eliminating contractible separating cycles,
we could simply try all possible ways $G_T$ appears as a subgraph of $G$ and for each of them verify whether
the faces in $Q_T$ contain a quadrangulation; this leads to a polynomial-time algorithm, but
with a rather high exponent.  However, there is a faster algorithm (described in more detail in~\cite{cylflow}):
First, we preprocess the graph $G$:
\begin{itemize}
\item While $G$ contains a contractible separating $(\le\!5)$-cycle $C$, delete from $G$ all vertices and edges
contained in the open disk $\Lambda$ bounded by $C$.
\item While $G$ contains a contractible separating $6$-cycle $C$ such that the open disk $\Lambda$ bounded by $C$
contains a face of length greater than four, delete from $G$ all vertices and edges contained in $\Lambda$.
\item While $G$ contains a contractible separating $7$-cycle $C$ such that the open disk $\Lambda$ bounded by $C$
contains either a face of length greater than five or at least two faces of length five, delete from $G$ all vertices
and edges contained in $\Lambda$.
\end{itemize}
These deletions do not change 3-colorability of $G$, by standard precoloring extension arguments~\cite{aksenov,bor7c}.
If the census of $G$ after performing these reductions is not $\emptyset$, $\{5,5\}$, $\{5,5,5,5\}$, $\{5,5,6\}$, or $\{5,7\}$,
then $G$ does not contain a subgraph with such a census (e.g., if $H\subseteq G$ had census $\{5,7\}$, then the $4$-faces
and the $5$-face of $H$ would also be faces in $G$ due to the first reduction rule, and the $7$-face of $H$ would either
be a face in $G$ or would be filled by $4$-faces and one $5$-face of $G$ as a consequence of the third reduction rule,
implying the census of $G$ is $\{5,7\}$ or $\{5,5\}$).  Consequently, by Theorem~\ref{thm-centor}, $G$ would not have a $4$-critical
subgraph, and thus it would be $3$-colorable.  Furthermore, by Corollary~\ref{cor-ew} and the first reduction rule, if $G$ has edge-width
at least five, it is either isomorphic to $I_4$ or $3$-colorable.

Therefore, the only remaining case is that the census of $G$ is $\emptyset$, $\{5,5\}$, $\{5,5,5,5\}$, $\{5,5,6\}$, or $\{5,7\}$
and the edge-width of $G$ is four.  In this case, we can apply an algorithm of~\cite{cylflow}, which decides whether $G$ is $3$-colorable
by at most 18 runs of a maximum flow algorithm (followed by easy postprocessing) in linear time.

Can we also find a 3-coloring when one exists?  The proof of Thomassen~\cite{thom-torus} that triangle-free toroidal
graphs without contractible 4-cycles are 3-colorable is quite involved, and by necessity has to be considered in this context.
But, what if we allow ourselves to consider it as a blackbox providing a 3-coloring of such a graph?

There is a simple reduction in the case $G$ contains a 4-face $f$ such that neither of
the two possible identifications creates a triangle: Note that in any 3-coloring of $G$, there are two opposite vertices on $f$
that receive the same color, and thus the graph obtained by identifying them is also $3$-colorable.  Hence, we can consider
both possible identifications, find one resulting in a $3$-colorable graph $G'$ using the $3$-colorability testing algorithm,
recursively find a $3$-coloring of $G'$, then transform it into a $3$-coloring of $G$ by giving both identified vertices the
color of the vertex created by their identification.

This reduction clearly applies when $G$ has edge-width at least six.  Moreover, in the case $G$ has edge-width at most five
and bounded census, we can obtain a $3$-coloring by the algorithm of~\cite{cylflow}.  However, the case that $G$ has
edge-width at most five (or, more precisely, no 4-face can be collapsed in both ways without creating a triangle)
and large census is problematic.  It can be dealt with, see~\cite{trfree7} for details, but the
resulting algorithm is not simple.

\section{Preliminaries}

For us, a \emph{multiset} is a set of distinct elements, each of which is
assigned a value (the values do not need to be necessarily distinct).  So, for
example, $\{5,5,6\}$ denotes a set of three elements, two of value $5$ and one
of value $6$. A union of multisets is always a disjoint union. So, for example,
$\{5,6\} \cup \{5\}$ is always $\{5,5,6\}$. 

All graphs we consider come with a fixed drawing in a surface. For any graph
$G$ we denote as $F(G)$ the sets of its faces. Suppose $X$ is a set of faces of
a graph drawn in a surface.  The \emph{census} of $X$ is the multiset of
lengths of the faces of $X$ whose length is not four, and the census $\CC(G)$ of a graph $G$
is the census of $F(G)$.

Consider a graph $H$ and a subgraph $C$ of $H$.  We say $H$ is \emph{$C$-critical}
if $C\neq H$ and for every proper subgraph $H'$ of $H$ containing $C$,
there exists a $3$-coloring of $C$ that extends to a $3$-coloring of $H'$, but
not to a $3$-coloring of $H$.  In other words, deleting any edge of
$E(H) \setminus E(C)$ from $H$ enlarges the set of precolorings of $C$ that extend
to a $3$-coloring of the whole graph.

Let $\mathcal{G}_k$ denote the set of all triangle-free graphs $H$ drawn
in the disk with the boundary traced by a cycle $C\subset H$ of length $k$
such that $H$ is $C$-critical.  Let $\SS_k$ denote set of all censuses of graphs in $\mathcal{G}_k$. 

The results of~\cite{thom-torus} and~\cite{trfree4} imply the following lemma, see~\cite{torirr} for more details.
\begin{lemma}\label{lemma-sset}
For every integer $k\ge 4$, the set $\SS_k$ is finite, and if $I\in \SS_k$, then $I=\emptyset$ or $\max I\le k-2$.  Furthermore,
\begin{itemize}
\item $\SS_4 = \SS_5 = \emptyset$,
\item $\SS_6 = \{\emptyset\}$,
\item $\SS_7 = \{\{5\}\}$,
\item $\SS_8 = \{\emptyset, \{5,5\}, \{6\}\}$, and
\item $\SS_9 = \{\{5\}, \{5,5,5\}, \{5,6\}, \{7\}\}$.
\end{itemize}
\end{lemma}

Consider a graph $G$ drawn in the torus and a $4$-face $f$ of $G$ bounded by the cycle $v_1v_2v_3v_4$.
Let $G'$ be the graph drawn in the torus obtained from $G$ by identifying $v_2$ with $v_4$
and suppressing the parallel edges from the resulting vertex $v$ to $v_1$ and $v_3$.
Let $P$ be the path $v_1vv_3$ in $G'$.  We say that $G'$ is obtained from $G$ by \emph{collapsing $f$ to $P$}.
If $G'$ is $3$-colorable then $G$ is $3$-colorable as well; hence, if $G$ is $4$-critical,
then $G'$ is not $3$-colorable, and thus it contains a $4$-critical subgraph $H$.
If $H$ is triangle-free, we say that $H$ is a \emph{reduction} of $G$.  In~\cite{torirr}, we proved the following technical lemma
relating $H$ to a subgraph of $G$ (there are several cases depending on whether the edges $vv_1$ and $vv_3$ and the vertices
$v_1$ and $v_3$ belong to $H$ or not), see Figure~\ref{fig-unredu} for an illustration.

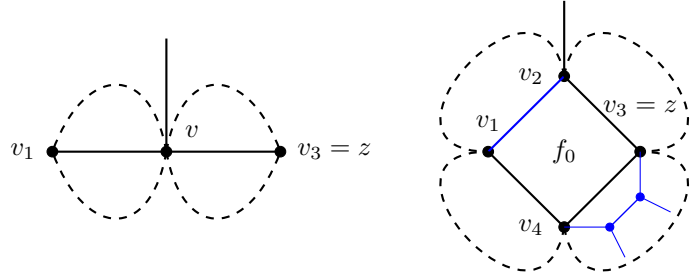
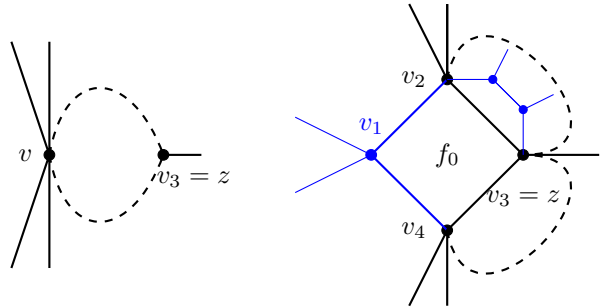
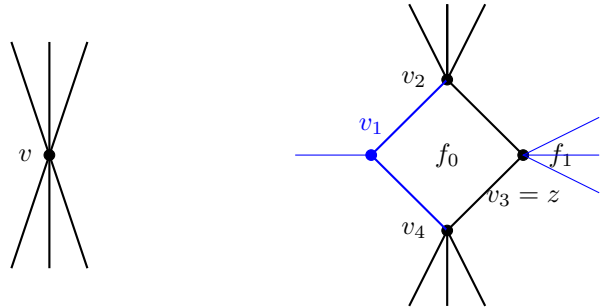
\begin{figure}
\renewcommand{\thesubfigure}{\roman{subfigure}}
\begin{subfigure}{0.9\textwidth}
\centering
\begin{tikzpicture}
\useasboundingbox (-2,-2) rectangle (2,2);
\coordinate (V1) at (-1.5,0);
\coordinate (V) at (0,0);
\coordinate (V3) at (1.5,0);
\coordinate (T) at (0,1.5);
\filldraw (V1) circle (2pt) node[left=1mm] {$v_1$};
\filldraw (V) circle (2pt) node[label=30:$v$] {}; 
\filldraw (V3) circle (2pt) node[right=1mm] {$v_3=z$};
\draw [thick] (V1) -- (V) -- (V3);
\draw [thick] (V) -- (T);
\draw [thick, dashed] (V) to[out=100,in=70, distance=35pt] (V1);
\draw [thick, dashed] (V) to[out=-100,in=-70, distance=35pt] (V1);
\draw [thick, dashed] (V) to[out=80,in=110, distance=35pt] (V3);
\draw [thick, dashed] (V) to[out=-80,in=-110, distance=35pt] (V3);
\end{tikzpicture}
\hspace{1cm}
\begin{tikzpicture}
\useasboundingbox (-2,-2) rectangle (2,2);
\coordinate (V1) at (-1,0);
\coordinate (V2) at (0,1);
\coordinate (V3) at (1,0);
\coordinate (V4) at (0,-1);
\coordinate (T) at (0,2);
\coordinate (F1) at (1,-0.6);
\coordinate (F2) at (0.6,-1);
\filldraw (V1) circle (2pt) node[above=1.5mm] {$v_1$};
\filldraw (V2) circle (2pt) node[left=1.5mm] {$v_2$};
\filldraw (V3) circle (2pt) node[above=3.5mm] {$v_3=z$};
\filldraw (V4) circle (2pt) node[left=1.5mm] {$v_4$};
\node at (0,0) {$f_0$};
\filldraw [color=blue] (F1) circle (1.5pt);
\filldraw [color=blue] (F2) circle (1.5pt);
\draw [thick, color=blue] (V1) -- (V2); 
\draw [thick] (V2) -- (V3) -- (V4) -- (V1);
\draw [thick] (V2) -- (T);
\draw [thick, dashed] (V2) to[out=100,in=180, distance=50pt] (V1);
\draw [thick, dashed] (V2) to[out=80,in=0, distance=50pt] (V3);
\draw [thick, dashed] (V4) to[out=-100,in=180, distance=50pt] (V1);
\draw [thick, dashed] (V4) to[out=-80,in=-0, distance=50pt] (V3);
\draw [color=blue] (V3) -- (F1) -- (F2) -- (V4);
\draw [color=blue] (1.4,-0.8) -- (F1) -- (F2) -- (0.8,-1.4); 
\end{tikzpicture}
\caption{$vz \in E(H)$ for some $z \in \{v_1,v_3\}$}
\end{subfigure}

\bigskip

\begin{subfigure}{0.9\textwidth}
\centering
\begin{tikzpicture}
\useasboundingbox (-2,-2) rectangle (2,2);
\coordinate (V) at (0,0);
\coordinate (V3) at (1.5,0);
\coordinate (TL) at (-0.5,1.5);
\coordinate (TC) at (0,1.5);
\coordinate (BL) at (-0.5,-1.5);
\coordinate (BC) at (0,-1.5);
\coordinate (RC) at (2,0);
\filldraw (V) circle (2pt) node[left=1mm] {$v$}; 
\filldraw (V3) circle (2pt) node[right=4mm,below=1mm] {$v_3=z$};
\draw [thick, dashed] (V) to[out=80,in=110, distance=35pt] (V3);
\draw [thick, dashed] (V) to[out=-80,in=-110, distance=35pt] (V3);
\draw [thick] (TL) -- (V) -- (TC); 
\draw [thick] (BL) -- (V) -- (BC); 
\draw [thick] (V3) -- (RC); 
\end{tikzpicture}
\hspace{1cm}
\begin{tikzpicture}
\useasboundingbox (-2,-2) rectangle (2,2);
\coordinate (V1) at (-1,0);
\coordinate (V2) at (0,1);
\coordinate (V3) at (1,0);
\coordinate (V4) at (0,-1);
\coordinate (F1) at (1,0.6);
\coordinate (F2) at (0.6,1);
\coordinate (TL) at (-0.5,2);
\coordinate (TC) at (0,2);
\coordinate (BL) at (-0.5,-2);
\coordinate (BC) at (0,-2);
\coordinate (RC) at (2,0);
\coordinate (LT) at (-2,0.5);
\coordinate (LB) at (-2,-0.5);
\filldraw (V1) [color=blue] circle (2pt) node[above=1.5mm] {$v_1$};
\filldraw (V2) circle (2pt) node[left=1.5mm] {$v_2$};
\filldraw (V3) circle (2pt) node[below=3.5mm] {$v_3=z$};
\filldraw (V4) circle (2pt) node[left=1.5mm] {$v_4$};
\node at (0,0) {$f_0$};
\filldraw [color=blue] (F1) circle (1.5pt);
\filldraw [color=blue] (F2) circle (1.5pt);
\draw [thick, color=blue] (V4) -- (V1) -- (V2); 
\draw [thick] (V2) -- (V3) -- (V4);
\draw [thick] (V2) -- (T);
\draw [thick, dashed] (V2) to[out=80,in=0, distance=50pt] (V3);
\draw [thick, dashed] (V4) to[out=-80,in=-0, distance=50pt] (V3);
\draw [thick] (TL) -- (V2) -- (TC);
\draw [thick] (BL) -- (V4) -- (BC);
\draw [thick] (V3) -- (RC);
\draw [color=blue] (LT) -- (V1) -- (LB);
\draw [color=blue] (V3) -- (F1) -- (F2) -- (V2);
\draw [color=blue] (1.4,0.8) -- (F1) -- (F2) -- (0.8,1.4); 
\end{tikzpicture}
\caption{$vv_1,vv_3 \notin E(H)$, $z \in V(H)$ for some $z \in \{v_1,v_3\}$}
\end{subfigure}

\bigskip

\begin{subfigure}{0.9\textwidth}
\centering
\begin{tikzpicture}
\useasboundingbox (-2,-2) rectangle (2,2);
\coordinate (V) at (0,0);
\coordinate (V3) at (1.5,0);
\coordinate (TL) at (-0.5,1.5);
\coordinate (TC) at (0,1.5);
\coordinate (TR) at (0.5,1.5);
\coordinate (BL) at (-0.5,-1.5);
\coordinate (BC) at (0,-1.5);
\coordinate (BR) at (0.5,-1.5);
\filldraw (V) circle (2pt) node[left=1mm] {$v$}; 
\draw [thick] (TL) -- (V) -- (TR); 
\draw [thick] (BL) -- (V) -- (BR); 
\draw [thick] (TC) -- (V) -- (BC); 
\end{tikzpicture}
\hspace{1cm}
\begin{tikzpicture}
\useasboundingbox (-2,-2) rectangle (2,2);
\coordinate (V1) at (-1,0);
\coordinate (V2) at (0,1);
\coordinate (V3) at (1,0);
\coordinate (V4) at (0,-1);
\coordinate (F1) at (1,0.6);
\coordinate (F2) at (0.6,1);
\coordinate (TL) at (-0.5,2);
\coordinate (TR) at (0.5,2);
\coordinate (TC) at (0,2);
\coordinate (BL) at (-0.5,-2);
\coordinate (BR) at (0.5,-2);
\coordinate (BC) at (0,-2);
\coordinate (RT) at (2,0.5);
\coordinate (RB) at (2,-0.5);
\coordinate (RC) at (2,0);
\coordinate (LC) at (-2,0);
\filldraw (V1) [color=blue] circle (2pt) node[above=1.5mm] {$v_1$};
\filldraw (V2) circle (2pt) node[left=1.5mm] {$v_2$};
\filldraw (V3) circle (2pt) node[below=3.5mm] {$v_3=z$};
\filldraw (V4) circle (2pt) node[left=1.5mm] {$v_4$};
\node at (0,0) {$f_0$};
\node at (1.5,0) {$f_1$};
\draw [thick, color=blue] (V4) -- (V1) -- (V2); 
\draw [thick] (TC) -- (V2) -- (V3) -- (V4) -- (BC);
\draw [thick] (V2) -- (T);
\draw [thick] (TL) -- (V2) -- (TR);
\draw [thick] (BL) -- (V4) -- (BR);
\draw [color=blue] (V3) -- (RC);
\draw [color=blue] (LC) -- (V1);
\draw [color=blue] (RT) -- (V3) -- (RB);
\end{tikzpicture}
\caption{$v_1,v_3 \notin V(H)$}
\end{subfigure}

\bigskip

\caption{A subgraph $G_1$ of $G$ (on the right side) corresponding to a reduction $H$ of $G$ (on the left side).  Blue edges and vertices
indicate parts of $G$ not belonging to $G_1$.}\label{fig-unredu}
\end{figure}

\begin{lemma}\label{lemma-unredu}
Let $G$ be a $4$-critical triangle-free graph drawn in the torus and let $H$ be a $4$-critical subgraph of a graph
obtained from $G$ by collapsing a $4$-face.   If $H$ is triangle-free, then there exists
\begin{itemize}
\item a subgraph $G_1$ of $G$ whose drawing in the torus is $2$-cell, and
\item a path $v_2zv_4$ contained in the boundary of a face $f_0$ of $G_1$, such that $f_0$ is not a face of $G$,
\end{itemize}
such that one of the following claims holds.
\begin{itemize}
\item[\textup{(i)}] $H$ is obtained from $G_1$ by identifying $v_2$ with $v_4$ to a new vertex $v$ within $f_0$
and suppresing the resulting $2$-face $vz$, or
\item[\textup{(ii)}] $H$ is obtained from $G_1$ by identifying $v_2$ with $v_4$ to a new vertex $v$ within $f_0$
and deleting both resulting edges between $v$ and $z$, or
\item[\textup{(iii)}] $z$ has degree two and it is incident with two distinct faces $f_0$ and $f_1$ in $G_1$,
$f_1$ is not a face of $G$, and $H$ is obtained from $G_1$ by contracting both edges incident with $z$.
\end{itemize}
\end{lemma}
Note that no parallel edges except for those explicitly mentioned in the statement of Lemma~\ref{lemma-unredu}
are created by the identification of $v_2$ with $v_4$ (in particular, when the collapsed $4$-face of $G$ is
bounded by a cycle $v_1v_2v_3v_4$ and $z=v_3$, at most one of the edges $v_1v_2$ and $v_1v_4$ belongs to $G_1$;
this is also the reason why we can assume $f_0$ is not a face of $G$).
Lemma~\ref{lemma-unredu} is useful in conjunction with the following basic property of critical graphs,
whose proof can be found e.g. in~\cite{trfree2}.

\begin{lemma}\label{SubgrCrit}
Let $G$ be a graph drawn in surface $\Sigma$, let $\Lambda\subseteq\Sigma$ be an open
disk whose boundary traces a closed walk in $G$, let $C$ be the subgraph of $G$ formed by
vertices and edges contained in the boundary of $\Lambda$, and let $Q$ be the subgraph of $G$ drawn in the closure of
$\Lambda$. If $G$ is 4-critical and $\Lambda$ is not its face, then $Q$ is $C$-critical.
\end{lemma}

Hence, $G$ can be obtained from its reduction $H$ by first finding its subgraph $G_1$ with properties
described in Lemma~\ref{lemma-unredu}, then filling some of the faces of $G_1$ by graphs
drawn in a disk and critical with respect to the boundary cycle.
As an example why this is useful, using the census information about critical graphs in disks
from Lemma~\ref{lemma-sset}, one can inductively prove Theorem~\ref{thm-centor}, see~\cite{torirr} for details.
For the purposes of this paper, we need to describe this inverse operation in more
detail in the setting of templates.

\section{Templates}

In the introduction, we defined simple templates as means of describing infinite families
of graphs drawn in the torus.  For the proof purposes, we need a more general notion.

\begin{definition}
A \emph{template} $T$ consists of a graph $G_T$ $2$-cell embedded in the torus
and a function $\theta_T$ assigning to each face of $G_T$ a multiset of integers greater or equal to five
such that $\sum \theta_T(f)\equiv |f|\pmod 2$ for every face $f\in F(G_T)$.

We say a graph $H$ $2$-cell embedded in the torus \emph{is represented} by the template $T$
if there exists a homeomorphism $\kappa$ of the torus mapping $G_T$ to a subgraph of $H$,
such that for each face $f$ of $G_T$,
$\theta_T(f)$ is equal to the census of the set of faces of $H$ drawn in $\kappa(f)$.
\end{definition}

We say a face $f\in F(G_T)$ is \emph{proper} if either $|f|=4$ and $\theta_T(f)=\emptyset$, or $\theta_T(f)=\{|f|\}$.
The motivation behind this definition is the following observation.
\begin{observation}
Let $G$ be a triangle-free $4$-critical graph represented by template $T$ via a homeomorphism $\kappa$.
If a face $f$ of $G_T$ is proper, then $\kappa(f)$ is a face of $G$.
\end{observation}
\begin{proof}
Let $k$ be the length of $f$.  Suppose for a contradiction $\kappa(f)$ is not a face of $G$,
and let $Q$ be the subgraph of $G$ drawn in the closure of $\kappa(f)$.
By Lemma~\ref{SubgrCrit}, we have $Q\in\mathcal{G}_k$, and thus the census of $Q$ belongs to $\SS_k$.
Since $\SS_4=\emptyset$, this implies $k>4$.
By Lemma~\ref{lemma-sset}, we have $\max(\theta_T(f))=\max(\CC(Q))\le k-2$, contradicting the assumption $f$ is proper.
\end{proof}
The simple templates we defined in the introduction thus correspond exactly to the templates $T$ such that
every face $f$ with $\theta_T(f)\neq\emptyset$ is proper; we will call the templates with this property \emph{direct}.
A template $T$ is \emph{relevant} if
\begin{itemize}
\item the graph $G_T$ is triangle-free,
\item $\bigcup_{f\in F(G_T)} \theta_T(f) \in \{\emptyset, \{5,5\}, \{5,5,5,5\},\{5,5,6\},\{5,7\}\}$, and
\item for every face $f$, $\theta_T(f) \in \SS_{|f|}\cup\{\{|f|\}\}$ if $|f| \neq 4$ and $\theta_T(f) = \emptyset$ if $|f| = 4$. 
\end{itemize}

By Lemma~\ref{SubgrCrit} and Theorem~\ref{thm-centor}, if a triangle-free
$4$-critical graph is represented by a template $T$, then $T$ is relevant.

\subsection{Winding numbers and 3-colorability of templates}\label{sec-winding}

Consider a graph with a proper coloring $\varphi$ by colors $\{0,1,2\}$.  For adjacent vertices $u$ and $v$, we define
$$\omega_\varphi(u,v)=\begin{cases}
1&\text{ if $\varphi(v)-\varphi(u)\equiv 1\pmod 3$}\\
-1&\text{ if $\varphi(v)-\varphi(u)\equiv -1\pmod 3$.}\\
\end{cases}$$
For a walk $W=v_0v_1\ldots v_m$, we define $\omega_\varphi(W)=\sum_{i=1}^m\omega_\varphi(v_{i-1},v_i)$.
Note that if $W$ is a cycle, then $\varphi$ can be extended in a natural way to a continuous mapping from the drawing of the cycle
to a triangle in the plane, and in this representation $\omega_\varphi(W)$ is equal to three times the winding number
of the corresponding closed curve around the interior of the triangle; thus, we will call $\omega_\varphi(W)$ the \emph{winding number}
of $\varphi$ on $W$ (ignoring the factor of three for convenience).

For a $2$-cell face $f$, let $W$ be the closed walk bounding $f$ in the clockwise direction, and define $\omega_\varphi(f)=\omega_\varphi(W)$.
Let us note two basic properties of the winding number.

\begin{observation}\label{obs-winsum}
Let $G$ be a graph $2$-cell embedded in an orientable surface and let $H$ be a subgraph of $G$.  Let $\varphi$ be a $3$-coloring of $G$.
Let $f$ be a $2$-cell face of $H$.  Then
$$\omega_\varphi(f)=\sum_{g\in F(H),g\subseteq f} \omega_\varphi(g).$$
Furthermore,
$$\sum_{f\in F(G)} \omega_\varphi(f)=0.$$
\end{observation}

\begin{observation}\label{obs-winvalues}
Let $\varphi$ be a $3$-coloring and let $W$ be a closed walk of length $m$.  Then
$3|\omega_\varphi(W)$, $\omega_\varphi(W)\equiv m\pmod 2$, and $|\omega_\varphi(W)|\le m$.
In particular, if $W$ has length $4$, then $\omega_\varphi(W)=0$.
\end{observation}

We need the following important result.
For a set $F$ of faces and an integer $k$, a \emph{winding number assignment summing to $k$} is a function $n:F\to\mathbb{Z}$ such that
$\sum_{f\in F} n(f)=k$ and for every $f\in F$, $n(f)$ is divisible by $3$, $n(f)$ has the same parity as $|f|$,
and $|n(f)|\le |f|$.

\begin{lemma}[{Dvořák and Lidický~\cite[a reformulation of Lemma~5]{col8cyc}}]\label{lemma-disk}
Let $G$ be a graph with a $2$-cell drawing in an orientable surface, let $H$ be a subgraph of $G$, and let $f$ be a $2$-cell face of $H$.
Let $\varphi$ be a $3$-coloring of $H$ and let $F$ be the set of faces of $G$ contained in $f$.
The coloring $\varphi$ does not extend to a $3$-coloring of
the subgraph of $G$ drawn in the closure of $f$ if and only if for every winding number assignment $n$ for $F$ summing to
$\omega_\varphi(f)$, either
\begin{itemize}
\item[\textup{(i)}] there exists a path $P$ in $G$ drawn in $f$ and intersecting the boundary of $f$ exactly in its endpoints, and denoting by $Q$ a part of the
clockwise boundary walk of $f$ between the endpoints of $P$ and by $F'\subseteq F$ the set of faces contained in the part of $f$ bounded by $Q+P$,
we have $$\Bigl|\omega_\varphi(Q)-\sum_{f\in F'} n(f)\Bigr|>|E(P)|;$$
or,
\item[\textup{(ii)}] there exists a cycle $C$ in $G$ drawn in $f$ and intersecting the boundary of $f$ in at most one point, and denoting
by $F'\subseteq F$ the set of faces contained in the part of $f$ bounded by $C$, we have
we have $$\Bigl|\sum_{f\in F'} n(f)\Bigr|>|C|.$$
\end{itemize}
\end{lemma}

For a multiset $I$ of positive integers, let us analogously define
a \emph{winding number assignment summing to $k$} as a function $n:I\to\mathbb{Z}$ such that $\sum_{i\in I} n(i)=k$
and for every $i\in I$, $n(i)$ is divisible by $3$, $n(i)$ has the same parity as $i$, and $|n(i)|\le i$.
Let $\omega(I)$ be the set of all integers $k$ such that there exists a winding number assignment for $I$ summing to $k$.
A \emph{proper $3$-coloring} of a template $T$ is a proper $3$-coloring $\varphi$ of $G_T$ such that $\omega_\varphi(f)\in\omega(\theta_T(f))$ for every face $f\in F(G_T)$.
A justification for this definition is the following fact.

\begin{lemma}\label{lemma-col}
Let $T$ be a template and let $\varphi$ be a proper $3$-coloring of $G_T$.  Then $\varphi$ is a proper $3$-coloring of $T$
if and only if there exists a graph $G$ represented by $T$ (via a homeomorphism $\kappa$) such that $\varphi\circ\kappa^{-1}$ extends to a proper $3$-coloring of $G$.
\end{lemma}
\begin{proof}
Suppose that there exists a graph $G$ represented by $T$ with a $3$-coloring $\psi$ extending $\varphi\circ\kappa^{-1}$, i.e., $\varphi(v)=\psi(\kappa(v))$ for every $v\in V(G_T)$.
Let $s=1$ if $\kappa$ preserves orientation and $s=-1$ if $\kappa$ reverses the orientation.
For any face $f\in F(G_T)$, let $W_f$ denote the closed walk of $G$ bounding $\kappa(f)$ in the clockwise direction and let $F_f(G)$ denote the set of faces of $G$ contained in $\kappa(f)$.
By Observation~\ref{obs-winsum}, we have
$$\omega_\varphi(f)=s\cdot\omega_\psi(\kappa(W_f))=s\cdot \sum_{g\in F_f(G)} \omega_\psi(g).$$
Since $\theta_T(f)$ is the census of $F_f(G)$, Observation~\ref{obs-winvalues} implies $\omega_\varphi(f)\in\omega(\theta_T(f))$.  Hence, $\varphi$ is a proper $3$-coloring of $T$.

The converse is proved by filling in the faces of $G_T$ by a suitably generic subgraphs (avoiding short paths between points on the boundary and short separating cycles)
with faces of appropriate census, so that Lemma~\ref{lemma-disk} implies $\varphi$ extends to a $3$-coloring of these subgraphs.  As we will not need this implication, the details are left to the reader.
\end{proof}

We say that a template $T$ is \emph{$3$-colorable} if it has a proper $3$-coloring.  Note that not all realizations of a $3$-colorable template are $3$-colorable,
but the converse is true, as is easy to see from Lemma~\ref{lemma-col}.

\begin{corollary}\label{cor-col}
If $G$ is represented by a template $T$ and $T$ is not $3$-colorable, then $G$ is not $3$-colorable.
\end{corollary}

Testing whether a template is $3$-colorable is of course NP-hard, but for reasonably small templates,
a brute-force algorithm enumerating all proper $3$-colorings of $G_T$ and testing the winding number conditions afterwards is fast enough for our purposes.

\subsection{Operations on templates}

Throughout the rest of this section, let $T$ be a template.

\subsubsection*{Hiding, revealing, subtemplates}
Let $e$ be an edge of $G_T$, where either $e$ is incident with two distinct faces or with a vertex $v$ of degree one.
Then $T\diamond e$ denotes a template obtained from $T$ as follows: in the former case $G_{T\diamond e}=G_T-e$,
in the latter case $G_{T\diamond e}=G_T-v$.
Let $f_e$ denote the face of $G_T-e$ in which $e$ used to be drawn, and let $X$ be the set of (at most two) faces of $G_T$ incident with $e$.
The function $\theta_{T\diamond e}$ matches $\theta_T$ on $F(G_T)\setminus X$, and $\theta_{T\diamond e}(f_e)=\bigcup \theta_T(X)$.
We say that $T\diamond e$ is obtained from $T$ by \emph{hiding} the edge $e$.  Conversely, we say that $T$ is obtained from $T\diamond e$ by \emph{revealing the edge $e$}.
Note that revealing edge may add new vertex of degree one into the template.

A template $T'$ is a \emph{subtemplate} of $T$ if a template homeomorphic to $T'$ is obtained from $T$ by repeatedly hiding edges.
Equivalently, there exists a homeomorphism $\kappa$ of the torus mapping $G_{T'}$
to a subgraph of $G_T$, such that for each face $f'$ of $G_{T'}$, we have
$$\theta_{T'}(f')=\bigcup_{f\in F(G_{T}),f\subseteq\kappa(f')} \theta_T(f).$$
Note that if a graph is represented by $T$, then it is also represented by $T'$.

\subsubsection*{Splitting}
Let $I$ be a multiset of integers.
We say a multiset $A$ is obtained from $I$ by \emph{splitting}
if $A$ is obtained from $I$ by
\begin{itemize}
\item removing an element of value $6$, or
\item replacing an element of value $i\ge 7$ by an element of value $i-2$, or
\item replacing an element of value $i\ge 8$ by two elements of values $i_1$ and $i_2$ such that $i_1,i_2\ge 5$ and $i_1+i_2=i+2$.
\end{itemize}
For a face $f\in F(G_T)$ with $\theta_T(f)\neq\emptyset$ and $\max\theta_T(f)\ge 6$, we say a template $T'$ is obtained from $T$ by
\emph{splitting inside $f$} if $G_{T'}=G_T$, $\theta_{T'}(g)=\theta_T(g)$ for $g\in F(G_T)\setminus\{f\}$,
and $\theta_{T'}(f)$ is obtained from $\theta_T(f)$ by splitting.  Note that if $H$ is a triangle-free graph represented by $T$ via a homeomorphism $\kappa$
and we add a chord to a face $h$ of $H$ (splitting it into two faces) so that the resulting graph $H'$ is triangle-free, then
$H'$ is represented by a template obtained from $T$ by splitting inside $f$, where $f$ is the face of $G_T$ such that $h\subseteq\kappa(f)$.

\subsubsection*{Filling}
We say a multiset $A$ is a \emph{filling} of a multiset $I$ if $A$ is obtained from $I$ by replacing each element $i\in I$ by
the elements of a multiset belonging to $\{\{i\}\}\cup \SS_i$.
We say that a template $T'$ is obtained from $T$ by \emph{filling} if $G_{T'}=G_T$
and $\theta_{T'}(f)$ is a filling of $\theta_T(f)$ for $f\in F(G_T)$.
By the definition of the sets $\SS_i$ and by Lemma~\ref{SubgrCrit}, if $H$ is a graph represented by $T$
and $H'$ is a $4$-critical triangle-free supergraph of $H$ in the torus, then $H'$ is represented
by a filling of $T$.

\subsubsection*{Boosting}
We say a multiset $A$ is obtained from a multiset $I$ by \emph{boosting} if either $A=I$ or
$A$ is obtained from $I$ by replacing an element of value $i$ by the elements of some multiset from $\SS_{i+2}$.
This operation has the following interpretation.  Let $G$ be a 4-critical triangle-free graph drawn in the torus.
Suppose $G_1$ is a subgraph of $G$ with a $2$-cell drawing in the torus and
$v_2zv_4$ is a path in the boundary of a face $f_0$ of $G_1$.
Let $H$ be the graph obtained from $G_1$ by identifying $v_2$ with $v_4$ within $f_0$ to a new vertex $v$ and suppressing the resulting $2$-face $vz$
and let $f^H_0$ be the face of $H$ corresponding to $f_0$, of length $k$.
Suppose $H$ is represented by a template $T$ via a homeomorphism $\kappa$ mapping an edge $v'z'\in E(G_T)$ to $vz$
and let $f'_0$ be the face of $G_T$ such that $f^H_0\subseteq \kappa(f'_0)$.
Let $T_1$ be the template obtained from $T$ as follows: $G_{T_1}$ is the graph created from $G_T$ by adding an edge $e'$ parallel to $v'z'$
and splitting the vertex $v'$ into two vertices $v'_2$ and $v'_4$ so that the $2$-face bounded by $e'$ and $v'z'$
merges with $f'_0$.  For any face $f\in F(G_{T_1})\setminus \{f'_0\}$, we let $\theta_{T_1}(f)=\theta_T(f)$.
If $f^H_0$ is a $4$-face, let $\theta_{T_1}(f'_0)=\theta_T(f'_0)\cup\{6\}$, otherwise
let $\theta_{T_1}(f'_0)=(\theta_T(f'_0)\setminus\{k\})\cup\{k+2\}$.
Then $G_1$ is represented by $T_1$ via the homeomorphism $\kappa'$.
Suppose now moreover that $f_0$ is not a face of $G$, and let $G_2$ be the subgraph of $G$ obtained from $G_1$ by adding
all vertices and edges of $G$ drawn in $f_0$.  Since $|f_0|=k+2$, Lemma~\ref{SubgrCrit} implies $G_2$ is represented by
a template $T_2$ obtained from $T_1$ by replacing $k+2$ in $\theta_{T_1}(f'_0)$ by the elements of some multiset from $\SS_{k+2}$.
Therefore, $\theta_{T_2}(f'_0)$ is obtained from $\theta_T(f'_0)$ by boosting.  Let us remark that if $k=4$, then since $\SS_6=\{\emptyset\}$,
we have $\theta_{T_2}(f'_0)=\theta_{T_1}(f'_0)\setminus\{6\}=\theta_T(f'_0)$, justifying the possibility of $A=I$ in the definition
of boosting.

\subsubsection*{Partial amplification}
Note that a vertex $v$ of a graph $G$ can appear in the boundary of a face $f$ several times.  The following definition
is used to indicate a particular incidence of $f$ with $v$.  We fix an open neighborhood $\delta$ of $v$ small enough so
that no other vertex appears in $\delta$, and for each edge $e$ intersecting $\delta$, $e$ is incident with $v$ and $e\cap \delta$ is an initial
segment of $e$ starting in $v$.  An \emph{angle} of $f$ at $v$ is an arcwise-connected subset $a$ of $\delta$ after removing the drawing of $G$ 
such that $a\subset f$.
A \emph{partial amplification} of $T$ is a template obtained in one of the following ways:
\begin{itemize}
\item[\textrm{(i-a)}] For some face $f\in F(G_T)$, we change $\theta_T(f)$ to a multiset obtained from it by boosting.
\item[\textrm{(i-b)}] For a vertex $v$ of $G_T$, we first either choose an edge $e$ incident with $v$, or reveal an edge $e$ incident with
$v$.  Then we choose an angle $a$ of a face $f$ at $v$, add an edge $e'$ parallel to $e$ so that $e$ and $e'$ bound a $2$-face $f'$,
and split the vertex $v$ into two vertices so that $f'$ merges with the angle $a$.  Finally, we change $\theta_T(f)$ to
a multiset obtained from it by boosting.
\item[\textrm{(ii-a)}] For some face $f\in F(G_T)$, we split a face inside $f$ and then change $\theta_T(f)$ to a multiset obtained from it by boosting.
\item[\textrm{(ii-b)}] For a vertex $v$ of $G_T$ and incident face $f$, we split a face inside $f$, then reveal an edge $e$ incident with $v$ and drawn in $f$.
Then we choose an angle $a$ of a face $f'$ at $v$, add an edge $e'$ parallel to $e$ so that $e$ and $e'$ bound a $2$-face $f''$,
and split the vertex $v$ into two vertices so that $f''$ merges with the angle $a$.  Finally, we change $\theta_T(f')$ to
a multiset obtained from it by boosting.
\item[\textrm{(iii-a)}] For some face $f\in F(G_T)$, we change $\theta_T(f)$ to a multiset obtained from it by boosting twice.
\item[\textrm{(iii-b)}] For a vertex $v$ of $G_T$ and distinct angles $a_1$ and $a_2$ of (not necessarily distinct) faces $g_1$ and $g_2$ at $v$,
we split $v$ into two vertices $v_2$ and $v_4$ and add a new vertex $z$ and a path $v_2zv_4$ so that the angles $a_1$ and $a_2$ now extend along this path.
Then we change $\theta_T(g_1)$ and $\theta_T(g_2)$ to multisets obtained from them by boosting (boosting twice when $g_1=g_2$).
\end{itemize}
Comparing this definition with Lemma~\ref{lemma-unredu} (the ``a'' cases corresponding to the situation where the vertex
$v$ discussed in the Lemma does not belong to $G_T$) and using the interpretations of the operations
of revealing an edge, splitting and boosting we introduced in this section, we conclude the following lemma holds.
\begin{lemma}\label{lemma-partex}
Let $G$ be a $4$-critical triangle-free graph drawn in the torus and let $H$ be a $4$-critical subgraph of a graph
obtained from $G$ by collapsing a $4$-face; suppose $H$ is triangle-free and let $G_1\subseteq G$, $f_0$ and possibly $f_1$ be as described
in Lemma~\ref{lemma-unredu}.  Let $G_2$ be the subgraph of $G$ obtained from $G_1$ by, for $i\in\{0,1\}$,
adding the vertices and edges of $G$ drawn in $f_i$.  If $H$ is represented by a template $T$, then
$G_2$ is represented by a partial amplification of $T$.
\end{lemma}
\begin{proof}
We give the argument in the case (ii) from the statement of Lemma~\ref{lemma-unredu}, the arguments in the remaining
two cases are similar.

Let $H'$ be the graph obtained from $G_1$ by identifying $v_2$ with $v_4$ to a new vertex $v$ within $f_0$
and let $g$ be its $2$-face bounded by the edges from $v$ to $z$.  Let $g_0$ be the face of $H'$ corresponding
to $f_0$ and let $a_0$ be the angle of $g_0$ such that $g$ merges with $a_0$ when we split $v$ back to $v_2$ and $v_4$.

Without loss of generality, assume the homeomorphism showing that $H$ is represented by $T$ is the identity.
Let $f$ be the face of $H$ in which the edges between $v$ and $z$ are drawn in $H'$.
A template $T'$ representing $H'$ is obtained from $T$ by splitting inside $f$
and in case that $v\in V(G_T)$, additionally revealing the edges $e$ and $e'$ between $v$ and $z$.

If $v\in V(G_T)$, then let $f'$ be the face of $G_{T'}$ and $a$ its angle containing the angle $a_0$.
As described in more detail after the definition of the boosting operation, a template representing $G_2$
is obtained from $T'$ by splitting the vertex $v$ into two vertices $v_2$ and $v_4$ so that the $2$-face bounded
by $e$ and $e'$ merges with the angle $a$, and changing $\theta_{T'}(f')$ to a multiset obtained from it by boosting.
This matches the case (ii-b) from the definition of partial amplification.
If $v\not\in V(G_T)$, and thus $v$ is drawn inside $f$, we achieve the same effect just by boosting the multiset $\theta_{T'}(f')$,
matching the case (ii-a).
\end{proof}

\subsubsection*{Amplification}

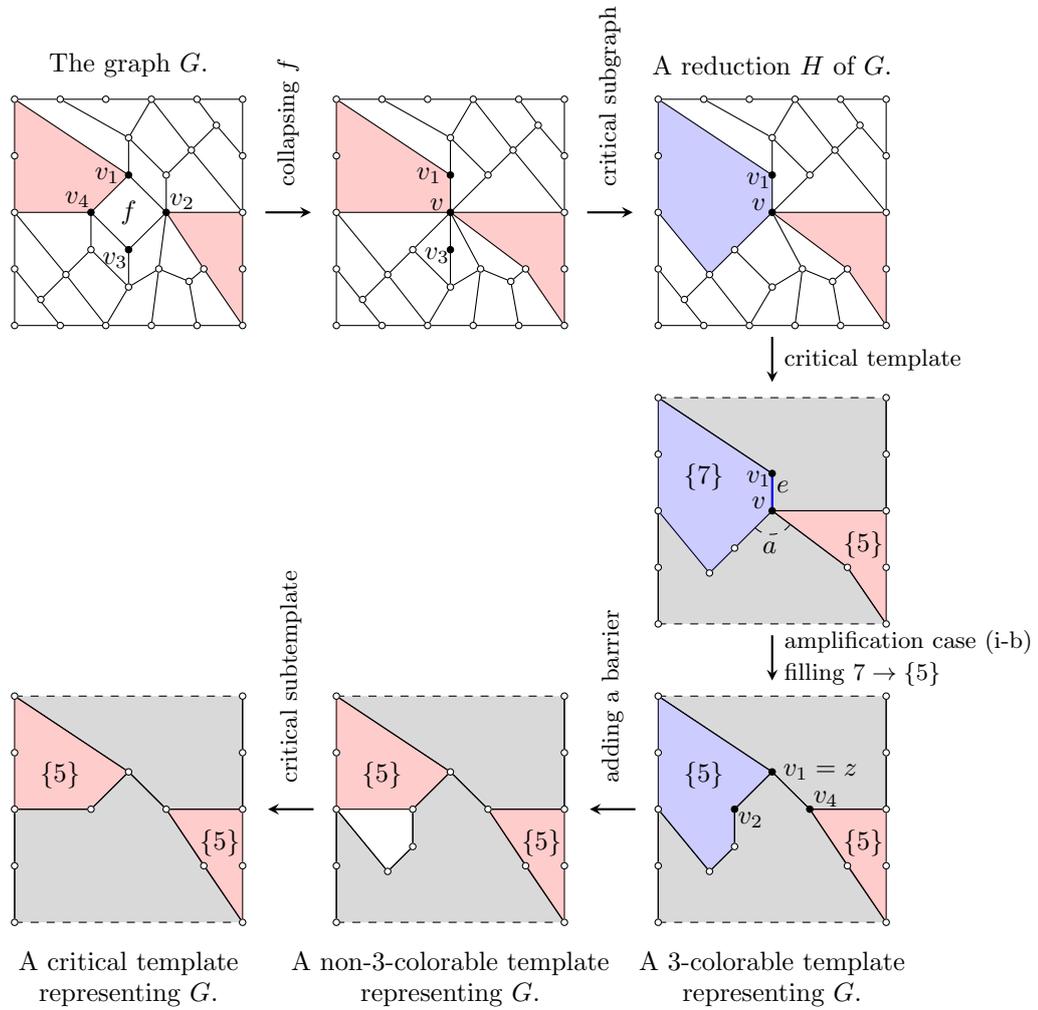
\begin{figure}
\newcommand{\vertices}{
\useasboundingbox (-10,-10) rectangle (10,10);
\coordinate (NW) at (-10,10);
\coordinate (N1) at (-6,10);
\coordinate (N2) at (-2,10); 
\coordinate (N3) at (2,10);
\coordinate (N4) at (6,10);
\coordinate (NE) at (10,10);
\coordinate (SW) at (-10,-10);
\coordinate (S1) at (-6,-10);
\coordinate (S2) at (-2,-10); 
\coordinate (S3) at (2,-10);
\coordinate (S4) at (6,-10);
\coordinate (SE) at (10,-10);
\coordinate (W1) at (-10,5);
\coordinate (W2) at (-10,0);
\coordinate (W3) at (-10,-5);
\coordinate (E1) at (10,5);
\coordinate (E2) at (10,0);
\coordinate (E3) at (10,-5);
\coordinate (FT) at (0,3.3);
\coordinate (FB) at (0,-3.3);
\coordinate (FL) at (-3.3,0);
\coordinate (FR) at (3.3,0);
\coordinate (FC) at (0,0);
\coordinate (A1) at (0,6.6);
\coordinate (A2) at (3.3,3.3);
\coordinate (A3) at (5.5,5.5);
\coordinate (A4) at (7.7,7.7);
\coordinate (B1) at (-7.7,-7.7);
\coordinate (B2) at (-5.5,-5.5);
\coordinate (B3) at (-3.3,-3.3);
\coordinate (B4) at (0,-6.6);
\coordinate (B5) at (2.65,-5);//(3.3,-3.3);
\coordinate (B6) at (5.3,-6.1);
\coordinate (B7) at (6.6,-5);
}
\begin{tikzpicture}[scale=0.15]
\vertices
\draw (NW)--(NE)--(SE)--(SW)--(NW);
\filldraw [color=red, opacity=0.2](NW)--(FT)--(FL)--(W2);
\draw (NW)--(FT)--(FL)--(W2);
\filldraw [color=red, opacity=0.2] (SE)--(B7)--(FR)--(E2);
\draw (SE)--(B7)--(FR)--(E2);
\draw (FT)--(FR)--(FB)--(FL);
\draw (N1)--(A1)--(A2) (FT)--(A1)--(N3)--(A3)--(E2) (FR)--(A2)--(A3)--(A4) (N4)--(A4)--(E1);
\draw (W2)--(B2)--(B3)--(FL) (W3)--(S1) (B1)--(B2)--(S2)--(B4)--(B5)--(B6)--(B7) (B3)--(B4)--(FB) (FR)--(B5)--(S3) (B6)--(S4);
\foreach \pt in {NW,N1,N2,N3,N4,NE,SW,S1,S2,S3,S4,SE,E1,E2,E3,W1,W2,W3,A1,A2,A3,A4,B1,B2,B3,B4,B5,B6,B7}{
	\filldraw (\pt) [color=white] circle (8pt);
	\draw (\pt) circle (8pt);
}
\foreach \pt in {FT,FB,FL,FR}{
	\filldraw (\pt) [color=black] circle (8pt);
}
\node at (FC) [] {$f$};
\node at (FT) [xshift=-8pt] {$v_1$};
\node at (FR) [xshift=6pt, yshift=4pt] {$v_2$};
\node at (FB) [xshift=-5pt, yshift=-4pt] {$v_3$};
\node at (FL) [xshift=-5pt, yshift=5pt] {$v_4$};
\node at (0,13) {The graph $G$.};
\draw [-stealth,thick] (12,0) to (16,0);
\node [label={[label distance=1mm]90:\rotatebox{90}{\small collapsing $f$}}] at (14,0){};
\end{tikzpicture}
\hspace{1cm}
\begin{tikzpicture}[scale=0.15]
\vertices
\draw (NW)--(NE)--(SE)--(SW)--(NW);
\filldraw [color=red, opacity=0.2](NW)--(FT)--(FC)--(W2);
\draw (NW)--(FT)--(FC)--(W2);
\filldraw [color=red, opacity=0.2] (SE)--(B7)--(FC)--(E2);
\draw (SE)--(B7)--(FC)--(E2);
\draw (FT)--(FC)--(FB);
\draw (N1)--(A1)--(A2) (FT)--(A1)--(N3)--(A3)--(E2) (FC)--(A2)--(A3)--(A4) (N4)--(A4)--(E1);
\draw (W2)--(B2)--(B3)--(FC) (W3)--(S1) (B1)--(B2)--(S2)--(B4)--(B5)--(B6)--(B7) (B3)--(B4)--(FB) (FC)--(B5)--(S3) (B6)--(S4);
\foreach \pt in {NW,N1,N2,N3,N4,NE,SW,S1,S2,S3,S4,SE,E1,E2,E3,W1,W2,W3,A1,A2,A3,A4,B1,B2,B3,B4,B5,B6,B7}{
	\filldraw (\pt) [color=white] circle (8pt);
	\draw (\pt) circle (8pt);
}
\foreach \pt in {FT,FB,FC}{
	\filldraw (\pt) [color=black] circle (8pt);
}
\node at (FT) [xshift=-8pt] {$v_1$};
\node at (FB) [xshift=-5pt, yshift=-2pt] {$v_3$};
\node at (FC) [xshift=-5pt, yshift=3pt] {$v$};
\draw [-stealth,thick] (12,0) to (16,0);
\node [label={[label distance=1mm]90:\rotatebox{90}{\small critical subgraph}}] at (14,0){};
\end{tikzpicture}
\hspace{1cm}
\begin{tikzpicture}[scale=0.15]
\vertices
\draw (NW)--(NE)--(SE)--(SW)--(NW);
\filldraw [color=blue, opacity=0.2](NW)--(FT)--(FC)--(B3)--(B2)--(W2);
\draw (NW)--(FT)--(FC);
\filldraw [color=red, opacity=0.2] (SE)--(B7)--(FC)--(E2);
\draw (SE)--(B7)--(FC)--(E2);
\draw (N1)--(A1)--(A2) (FT)--(A1)--(N3)--(A3)--(E2) (FC)--(A2)--(A3)--(A4) (N4)--(A4)--(E1);
\draw (W2)--(B2)--(B3)--(FC) (W3)--(S1) (B1)--(B2)--(S2)--(B4)--(B5)--(B6)--(B7) (B3)--(B4) (FC)--(B5)--(S3) (B6)--(S4);
\foreach \pt in {NW,N1,N2,N3,N4,NE,SW,S1,S2,S3,S4,SE,E1,E2,E3,W1,W2,W3,A1,A2,A3,A4,B1,B2,B3,B4,B5,B6,B7}{
	\filldraw (\pt) [color=white] circle (8pt);
	\draw (\pt) circle (8pt);
}
\foreach \pt in {FT,FC}{
	\filldraw (\pt) [color=black] circle (8pt);
}
\node at (FT) [xshift=-5pt, yshift=-2pt] {$v_1$};
\node at (FC) [xshift=-5pt, yshift=3pt] {$v$};

\node at (0,13) [align=left] {A reduction $H$ of $G$.};
\draw [-stealth,thick] (0,-11) to (0,-15);
\node at (0,-13) [right=1pt, align=left] {\small critical template};
\end{tikzpicture}
\newline
\vspace{0.5cm}
\newline
\begin{tikzpicture}[scale=0.15]
\useasboundingbox (-10,-10) rectangle (10,10);
\end{tikzpicture}
\hspace{1cm}
\begin{tikzpicture}[scale=0.15]
\useasboundingbox (-10,-10) rectangle (10,10);
\end{tikzpicture}
\hspace{1cm}
\begin{tikzpicture}[scale=0.15]
\vertices
\filldraw [fill=black!15!white] (NW)--(FT)--(FC)--(E2)--(NE);
\filldraw [fill=black!15!white] (SW)--(W2)--(B2)--(B3)--(FC)--(B7)--(SE);
\draw (NE)--(SE) (NW)--(SW);
\draw [dashed] (NW)--(NE) (SW)--(SE);
\filldraw [color=blue, opacity=0.2](NW)--(FT)--(FC)--(B3)--(B2)--(W2);
\draw (NW)--(FT)--(FC);
\filldraw [color=red, opacity=0.2] (SE)--(B7)--(FC)--(E2);
\draw (SE)--(B7)--(FC)--(E2);
\draw (FT)--(FC);
\draw [color=blue, thick] (FC)--(FT);
\foreach \pt in {NW,NE,SW,SE,E1,E2,E3,W1,W2,W3,B2,B3,B7}{
	\filldraw (\pt) [color=white] circle (8pt);
	\draw (\pt) circle (8pt);
}
\foreach \pt in {FT,FC}{
	\filldraw (\pt) [color=black] circle (8pt);
}
\node at (-6,3) {$\{7\}$};
\node at (8,-3) {$\{5\}$};
\node at (FT) [xshift=-5pt, yshift=-2pt] {$v_1$};
\node at (FC) [xshift=-5pt, yshift=3pt] {$v$};
\node at (FT) [xshift=4pt, yshift=-5pt] {$e$};
\draw [dashed] (-1.5,-1.5) arc (225:325:2cm);
\node at (FC) [xshift=-1pt, yshift=-14pt] {$a$};

\draw [-stealth,thick] (0,-11) to (0,-15);
\node at (0,-13) [right=1pt, align=left] {\small amplification case (i-b) \\ \small filling $7 \rightarrow \{5\}$};
\end{tikzpicture}
\newline
\vspace{0.5cm}
\newline
\begin{tikzpicture}[scale=0.15]
\vertices
\filldraw [fill=black!15!white] (NW)--(FT)--(FR)--(E2)--(NE);
\filldraw [fill=black!15!white] (SW)--(W2)--(FL)--(FT)--(FR)--(SE);
\draw (NE)--(SE) (NW)--(SW);
\draw [dashed] (NW)--(NE) (SW)--(SE);
\filldraw [color=red, opacity=0.2](NW)--(FT)--(FL)--(W2);
\draw (NW)--(FT)--(FL)--(W2);
\filldraw [color=red, opacity=0.2] (SE)--(B7)--(FR)--(E2);
\draw (SE)--(B7)--(FR)--(E2);
\draw (FT)--(FR);
\foreach \pt in {NW,NE,SW,SE,E1,E2,E3,W1,W2,W3,FT,FR,FL,B7}{
	\filldraw (\pt) [color=white] circle (8pt);
	\draw (\pt) circle (8pt);
}
\node at (-6,3) {$\{5\}$};
\node at (8,-3) {$\{5\}$};
\node at (0,-15) [align=center]{A critical template \\ representing $G$.};
\end{tikzpicture}
\hspace{1cm}
\begin{tikzpicture}[scale=0.15]
\vertices
\filldraw [fill=black!15!white] (NW)--(FT)--(FR)--(E2)--(NE);
\filldraw [fill=black!15!white] (SW)--(W2)--(B2)--(B3)--(FL)--(FT)--(FR)--(SE);
\draw (NE)--(SE) (NW)--(SW);
\draw [dashed] (NW)--(NE) (SW)--(SE);
\filldraw [color=red, opacity=0.2](NW)--(FT)--(FL)--(W2);
\draw (NW)--(FT)--(FL)--(W2);
\filldraw [color=red, opacity=0.2] (SE)--(B7)--(FR)--(E2);
\draw (SE)--(B7)--(FR)--(E2);
\draw (FT)--(FR);
\draw (W2)--(B2)--(B3)--(FL);
\foreach \pt in {NW,NE,SW,SE,E1,E2,E3,W1,W2,W3,FT,FR,FL,B2,B3,B7}{
	\filldraw (\pt) [color=white] circle (8pt);
	\draw (\pt) circle (8pt);
}
\node at (-6,3) {$\{5\}$};
\node at (8,-3) {$\{5\}$};
\node at (0,-15) [align=center]{A non-3-colorable template \\ representing $G$.};
\draw [-stealth,thick] (-12,0) to (-16,0);
\node [label={[label distance=1mm]90:\rotatebox{90}{\small critical subtemplate}}] at (-14,0){};
\end{tikzpicture}
\hspace{1cm}
\begin{tikzpicture}[scale=0.15]
\vertices
\filldraw [fill=black!15!white] (NW)--(FT)--(FR)--(E2)--(NE);
\filldraw [fill=black!15!white] (SW)--(W2)--(B2)--(B3)--(FL)--(FT)--(FR)--(SE);
\draw (NE)--(SE) (NW)--(SW);
\draw [dashed] (NW)--(NE) (SW)--(SE);
\filldraw [color=blue, opacity=0.2](NW)--(FT)--(FL)--(B3)--(B2)--(W2);
\draw (NW)--(FT)--(FL)--(B3)--(B2)--(W2);
\filldraw [color=red, opacity=0.2] (SE)--(B7)--(FR)--(E2);
\draw (SE)--(B7)--(FR)--(E2);
\draw (FT)--(FR);
\foreach \pt in {NW,NE,SW,SE,E1,E2,E3,W1,W2,W3,B2,B3,B7}{
	\filldraw (\pt) [color=white] circle (8pt);
	\draw (\pt) circle (8pt);
}
\foreach \pt in {FT,FR,FL}{
	\filldraw (\pt) [color=black] circle (8pt);
}
\node at (-6,3) {$\{5\}$};
\node at (8,-3) {$\{5\}$};
\node at (FT) [xshift=18pt] {$v_1=z$};
\node at (FR) [xshift=6pt, yshift=4pt] {$v_4$};
\node at (FL) [xshift=6pt, yshift=-4pt] {$v_2$};
\node at (0,-15) [align=center]{A 3-colorable template \\ representing $G$.};
\draw [-stealth,thick] (-12,0) to (-16,0);
\node [label={[label distance=1mm]90:\rotatebox{90}{\small adding a barrier}}] at (-14,0){};
\end{tikzpicture}
\newline
\vspace{1cm}
\caption{The process of deriving a critical template for a 4-critical triangle-free graph $G$ from a template for its reduction $H$.
Grey faces $g$ in templates are quadrangulated, $\theta(g)=\emptyset$.  Red and blue faces have length five and seven, respectively.}\label{fig-example}
\end{figure}

An \emph{amplification} of $T$ is a filling of a partial amplification of $T$.  As $G_2$ in Lemma~\ref{lemma-partex}
is a subgraph of $G$, a template representing $G$ is obtained from one representing $G_2$ by filling.
An example of the amplification operation is given on the right in Figure~\ref{fig-example}.

\begin{corollary}\label{cor-part}
Let $G$ be a $4$-critical triangle-free graph drawn in the torus and let $H$ be a $4$-critical subgraph of a graph
obtained from $G$ by collapsing a $4$-face.  If $H$ is triangle-free and $H$ is represented by a template $T$, then
$G$ is represented by an amplification of $T$.
\end{corollary}

However, note that even if $T$ is not $3$-colorable, some of its amplifications can be $3$-colorable.  Next, we deal with this issue.

\subsection{Making a template non-3-colorable}

The following definitions are motivated by Lemma~\ref{lemma-disk}(i), essentially restating what an obstruction
to extendability of a $3$-coloring may look like from the point of view of templates. 
Consider a template $T$ and let $f$ be a face of $G_T$.  A \emph{strut} of $f$ is
a directed path $P$ whose endpoints $s$ and $t$ are vertices of $G_T$ in the boundary of $f$, and the rest of $P$ is drawn inside $f$.
For a strut $P$, let $R(P)$ denote the subwalk of the clockwise boundary walk of $f$ starting in $t$ and ending in $s$,
and let $r(P)$ denote the part of $f$ bounded by the cycle formed by the concatenation of $P$ and $R(P)$.  Consider a proper $3$-coloring $\varphi$ of $T$
and a winding number assignment $n$ for $\theta_T(f)$ summing to $\omega_T(\varphi)$.  An \emph{$f$-barrier} for $(\varphi,n)$
is a pair $(P,I)$, where $P$ is a strut of $f$ and $I$ is a multisubset of $\theta_T(f)$ such that
$$\Bigl|\omega_\varphi(R(P))-\sum_{i\in I} n(i)\Bigr|>|E(P)|.$$
A set $B$ of $f$-barriers is \emph{drawing-consistent} if the intersection of the drawings of any two of the struts is a union of vertices and edges;
in such a case, let $\bigcup B$ denote the graph consisting of the union of the struts.
We say that $B$ \emph{blocks $\varphi$} if for every winding number assignment $n$ for $\theta_T(f)$ summing to $\omega_T(\varphi)$,
$B$ contains an $f$-barrier for $(\varphi,n)$.

Given a system $\BB=\{B_f:f\in F(G_T)\}$, where $B_f$ is a drawing-consistent set of $f$-barriers for each face $f\in F(G_T)$, 
a \emph{realization} of $\BB$ is a template $T'$ such that $$G_{T'}=G_T\cup\bigcup_{f\in F(G_T)}\bigcup B_f,$$
$T$ is a subtemplate of $T'$, i.e., every face $f\in F(G_T)$ satisfies
$$\theta_T(f)=\bigcup_{h\in F(G_T'),h\subseteq f} \theta_{T'}(h),$$
and for every $f\in F(G_T)$ and $(P,I)\in B_f$,
$$I=\bigcup_{h\in F(G_{T'}),h\subseteq r(P)} \theta_{T'}(h).$$
The last condition expresses that in the realization, the values from $I$ are exactly those assigned to the faces of the
realization contained in $r(P)$.  For this reason, a system $\BB$ does not necessarily have a realization
even when it is drawing-consistent, since it may not be possible to choose $\theta_{T'}$ so that the last condition holds.
On the other hand, it may also be possible to choose $\theta_{T'}$ (and thus a realization) in several different ways.
If $\BB$ has a realization, we say that it is \emph{consistent}.  For a proper $3$-coloring $\varphi$ of $T$, we say
that $\BB$ \emph{blocks $\varphi$} if there exists $f\in F(G_T)$ such that $B_f$ blocks $\varphi$.
The following claim is essentially clear from the definitions and Lemma~\ref{lemma-disk}.

\begin{lemma}\label{lemma-supnocol}
Let $T$ be a relevant template and let $\BB=\{B_f:f\in F(G_T)\}$ be a consistent system of sets of barriers.
Let $T'$ be a realization of $\BB$.  If $\BB$ blocks every proper $3$-coloring of $T$, then $T'$ is not
$3$-colorable.
\end{lemma}
\begin{proof}
Suppose for a contradiction $T'$ has a proper $3$-coloring $\varphi'$, and let $\varphi$ be the restriction
of $\varphi'$ to $V(G_T)$.  Consider any face $f$ of $G_T$, and let $F_f=\{g\in F(G_{T'}):g\subseteq f\}$.
Let $n_f:\bigcup_{g\in F_f}\theta_{T'}(g)\to\mathbb{Z}$ be a function whose restriction to $\theta_{T'}(g)$
is a winding number assignment for $\theta_{T'}(g)$ summing to $\omega_{\varphi'}(g)$ for every $g\in F_f$;
such a function $n_f$ exists, since $\varphi'$ is a proper $3$-coloring of $T'$.
Since $T'$ is a realization of $T$, Observation~\ref{obs-winsum} implies
\begin{align*}
\omega_\varphi(f)&=\omega_{\varphi'}(f)=\sum_{g\in F_f} \omega_{\varphi'}(g)\\
&=\sum_{g\in F_f}\sum_{i\in \theta_{T'}(g)} n_f(i)=\sum_{i\in\theta_T(f)} n_f(i),
\end{align*}
and thus $n_f$ is a winding number assignment for $\theta_T(f)$ summing to $\omega_\varphi(f)$.
Since this holds for every $f\in F(G_T)$, we conclude that $\varphi$ is a proper $3$-coloring of $T$.

Therefore, $\BB$ blocks $\varphi$, and thus for some $f\in F(G_T)$, $B_f$ contains an $f$-barrier $(P,I)$ for $(\varphi, n_f)$.
Since $\varphi'$ is a proper $3$-coloring of $T'$ and by Observation~\ref{obs-winsum}, we have
\begin{align*}
\omega_{\varphi'}(P+R(P))&=\sum_{g\in F(G_{T'}),g\subseteq r(P)} \omega_{\varphi'}(g)\\
&=\sum_{g\in F(G_{T'}),g\subseteq r(P)}\sum_{i\in \theta_{T'}(g)} n_f(i)\\
&=\sum_{i\in I} n_f(i)
\end{align*}
Since $|\omega_{\varphi'}(u,v)|\le 1$ for any adjacent $u,v\in V(G_{T'})$, we have $|\omega_{\varphi'}(P)|\le|E(P)|$.
Consequently
\begin{align*}
\Bigl|\omega_\varphi(R(P))-\sum_{i\in I} n(i)\Bigr|&=\Bigl|\omega_{\varphi'}(R(P))-\sum_{i\in I} n(i)\Bigr|\\
&\le \Bigl|\omega_{\varphi'}(P+R(P))-\sum_{i\in I} n(i)\Bigr|+|E(P)|=|E(P)|,
\end{align*}
which is a contradiction, since $(P,I)$ is an $f$-barrier for $(\varphi,n_f)$.
\end{proof}

More interestingly, a converse holds as well.

\begin{lemma}\label{lemma-barrier}
Let $T$ be a relevant template.  If a $4$-critical triangle-free graph $G$ is represented by $T$, then
there exists a consistent system $\BB$ of sets of barriers which blocks every proper $3$-coloring of $T$
such that $G$ is represented by a realization of $\BB$.
\end{lemma}
\begin{proof}
Without loss of generality, we can assume $T$ represents $G$ via the identity homeomorphism, and thus $G_T\subseteq G$.
For a face $f\in F(G_T)$, let $F_f$ denote the set of faces of $G$ contained in $f$, and let us fix a bijection $\gamma_f$
mapping each face $g\in F_f$ of length at least $5$ to an element of $\theta_T(f)$ of value $|g|$.

Consider a proper $3$-coloring $\varphi$ of $G_T$.  Since $G$ is not $3$-colorable, there exists $f\in F(G_T)$ such that
$\varphi$ does not extend to a $3$-coloring of the subgraph of $G$ drawn in $f$.
Consider any winding number assignment $n$ for $\theta_T(f)$ summing to $\omega_\varphi(f)$.  For $g\in F_f$,
let $n'(g)=n(\gamma_f(g))$ if $|g|>4$ and $n'(g)=0$ if $|g|=4$; then $n'$ is a winding number assignment for $F_f$ summing to $\omega_\varphi(f)$.
We now apply Lemma~\ref{lemma-disk} to $\varphi$ and $n'$.

Suppose first that (ii) holds; let $C$ and $F'$ be as in the statement, and let $I=\gamma_f(F')$, so that
$\Bigl|\sum_{i\in I} n(i)\Bigr|>|C|$.  Since $G$ is $4$-critical and triangle-free, Lemma~\ref{SubgrCrit}
implies $I\in\SS_{|C|}$.  If $I=\{i\}$, then since $n$ is a winding number assignment, we have
$|C|=|n(i)|\le i=\max I$, contradicting Lemma~\ref{lemma-sset}.
Consequently $|I|\ge 2$, and thus Lemma~\ref{lemma-sset} implies $|C|\ge 8$.
Since $n(5)\le 3$ and $\Bigl|\sum_{i\in I} n(i)\Bigr|>|C|\ge 8$, we have $I\neq\{5,5\}$, and thus Lemma~\ref{lemma-sset} implies $|C|\ge 9$.
Since $n(6)\le 6$ and $n(7)\le 3$, the same argument now gives $I\neq \{5,6\}$, $I\neq \{5,7\}$, and $I\neq\{5,5,5\}$.
Let $J=\bigcup_{h\in F(G_T)} \theta_T(h)$; since $T$ is relevant and $I\subseteq J$, we conclude that $J$ is either $\{5,5,5,5\}$
or $\{5,5,6\}$ and $I=J$.
Therefore,
$$\sum_{i\in I} n(i)=\sum_{h\in F(G_T)} \sum_{i\in \theta_T(h)} n(i)=\sum_{h\in F(G_T)}\omega_\varphi(h)=0$$
by Observation~\ref{obs-winsum}.  Therefore, $\Bigl|\sum_{i\in I} n(i)\Bigr|=0<|C|$, which is a contradiction.

Therefore, (i) holds; let $P$ be as in the statement (with $F'=\{g\in F(G):g\subseteq r(P)\}$)
and let $I=\gamma_f(F')$.  Then $(P,I)$ an $f$-barrier for $(\varphi,n)$.  Collecting all such barriers for all proper $3$-colorings $\varphi$ of $T$
and for all choices of $n$ gives us a consistent system $\BB$ of sets of barriers which blocks every proper $3$-coloring of $T$,
with a realization representing $G$.
\end{proof}

A system $\BB$ of sets of barriers is \emph{$T$-minimal} if $\BB$ blocks all proper $3$-colorings of $T$, but removing any barrier from any of the sets results in
a sytem that no longer blocks all proper $3$-colorings of $T$.  Note that a $T$-minimal system has bounded size, and thus there are (up to homeomorphism)
only finitely many $T$-minimal systems and their realizations.  Let us remark that a realization of a $T$-minimal set of barriers
may still contain a proper non-$3$-colorable subtemplate.

A template $T'$ is \emph{critical} if $T'$ is not $3$-colorable, but all proper subtemplates of $T'$ are $3$-colorable.
Observe that if a template is $3$-colorable, then all its subtemplates are also $3$-colorable.  Hence, a non-$3$-colorable template $T''$
has a critical subtemplate, which can be obtained from $T''$ by repeatedly hiding edges whose removal does not cause the template to become $3$-colorable.

We can now combine all the results.  We say that a template $T_3$ is \emph{grown from a template $T$}
if there exists an amplification $T_1$ of $T$, a $T_1$-minimal consistent system $\BB$ of sets of barriers,
and a realization $T_2$ of $\BB$ such that $T_3$ is a relevant critical subtemplate of $T_2$.
See the bottom part of Figure~\ref{fig-example} for an illustration.

\begin{theorem}\label{thm-aplif}
Let $G$ be a $4$-critical triangle-free graph drawn in the torus and let $H$ be a $4$-critical subgraph of a graph
obtained from $G$ by collapsing a $4$-face.  If $H$ is triangle-free and $H$ is represented by a template $T$, then
$G$ is represented by a template grown from $T$.
\end{theorem}
\begin{proof}
By Corollary~\ref{cor-part}, we can choose an amplification $T_1$ of $T$ representing $G$.
Since $G$ is $4$-critical and triangle-free, Theorem~\ref{thm-centor} and Lemma~\ref{SubgrCrit} imply the template $T_1$ is relevant.
By Lemma~\ref{lemma-barrier}, there exists a $T_1$-minimal consistent system $\BB$ of sets of barriers
such that $G$ is represented by a realization $T_2$ of $\BB$.
By Lemma~\ref{lemma-supnocol}, $T_2$ is not $3$-colorable, and thus it has a critical subtemplate $T_3$, which also represents $G$.
\end{proof}

\section{A complete description of 4-critical triangle-free graphs drawn in the torus}

Recall a $4$-critical triangle-free graph drawn in the torus is \emph{irreducible} if none of its $4$-faces can be collapsed
so that the resulting graph is triangle-free. In another paper, we have identified all irreducible graphs.

\begin{theorem}[Dvořák and Pekárek~\cite{torirr}]\label{thm-irr}
There are exactly four non-homeomorphic irreducible graphs drawn in the torus:
$I_4, I_5, I_7^a, I_7^b$, as depicted in Figure~\ref{fig-irreducibles}. 
\end{theorem}

\begin{figure}[h]
\centering
\begin{subfigure}{0.24\textwidth}
\includegraphics[width=90pt]{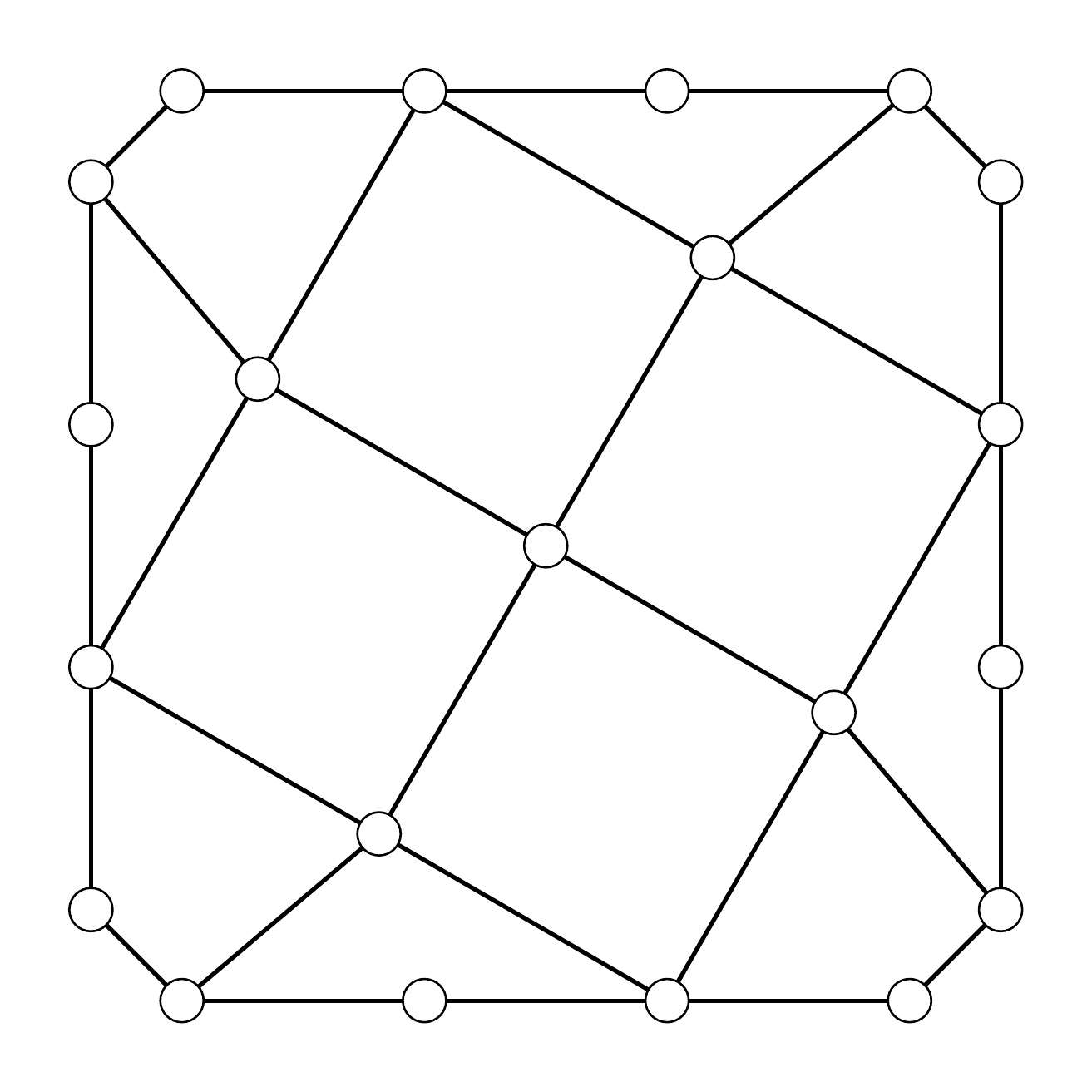}
\caption{$I_4$}
\end{subfigure}
\begin{subfigure}{0.24\textwidth}
\includegraphics[width=90pt]{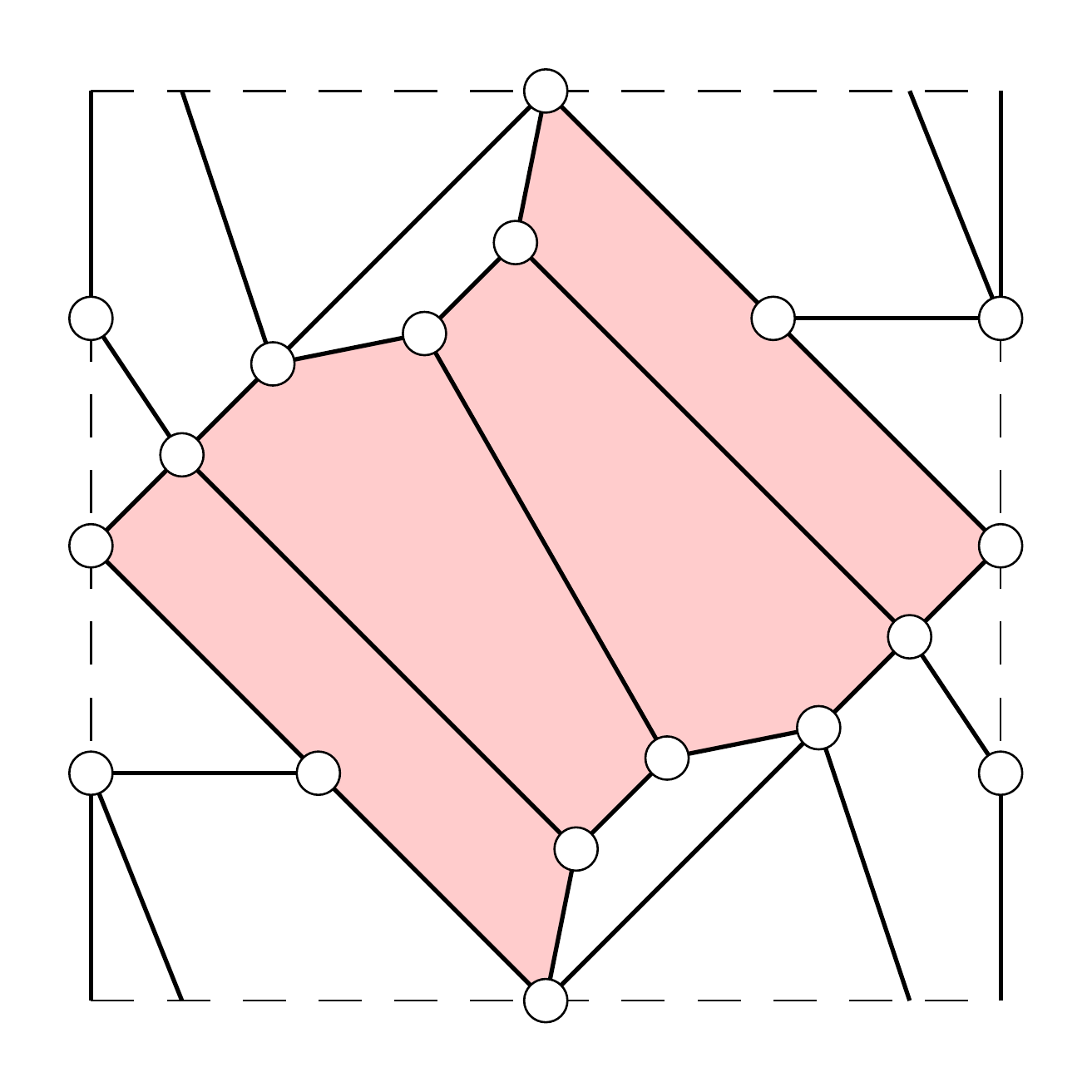}
\caption{$I_5$}
\end{subfigure}
\begin{subfigure}{0.24\textwidth}
\includegraphics[width=90pt]{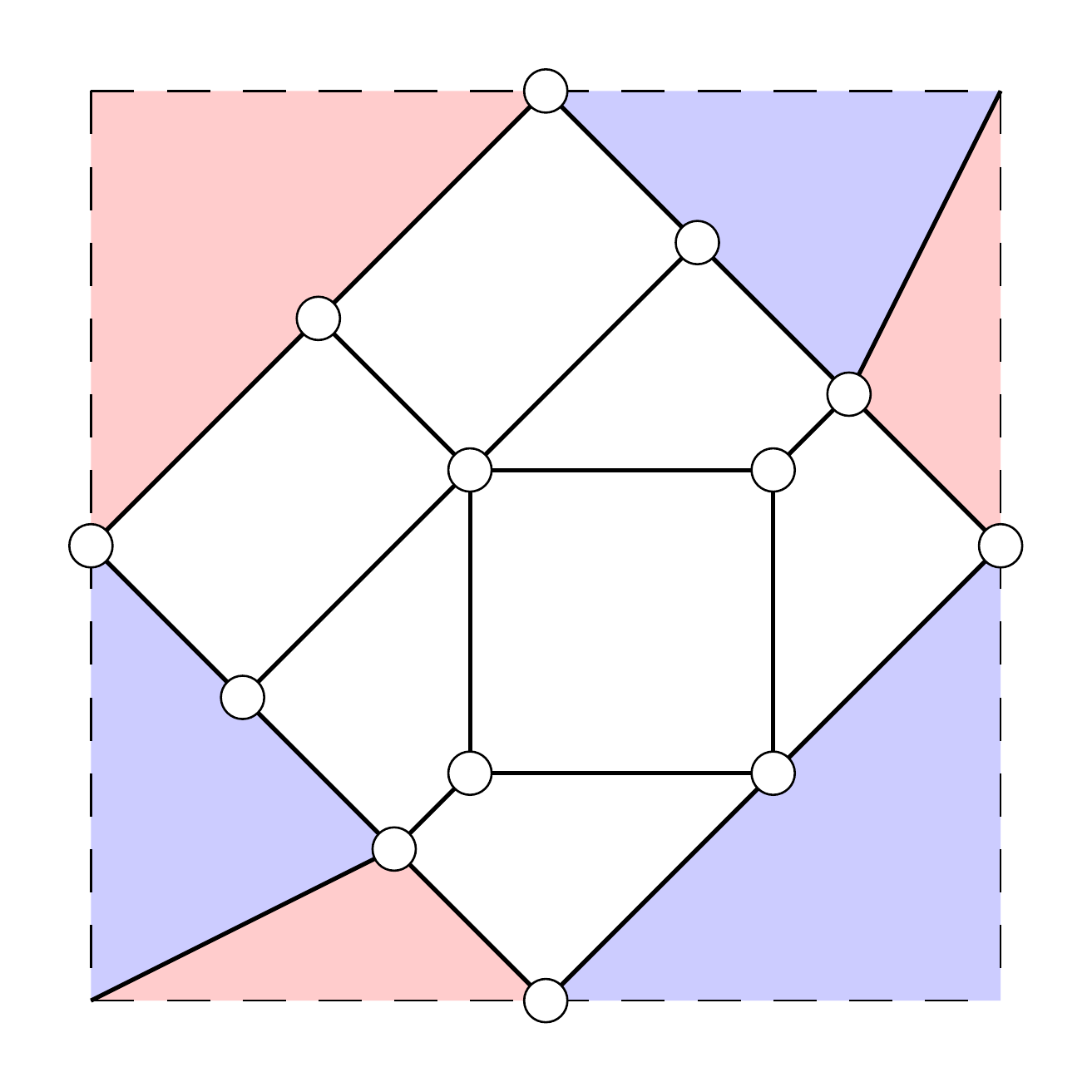}
\caption{$I_7^{a}$}
\end{subfigure}
\begin{subfigure}{0.24\textwidth}
\includegraphics[width=90pt]{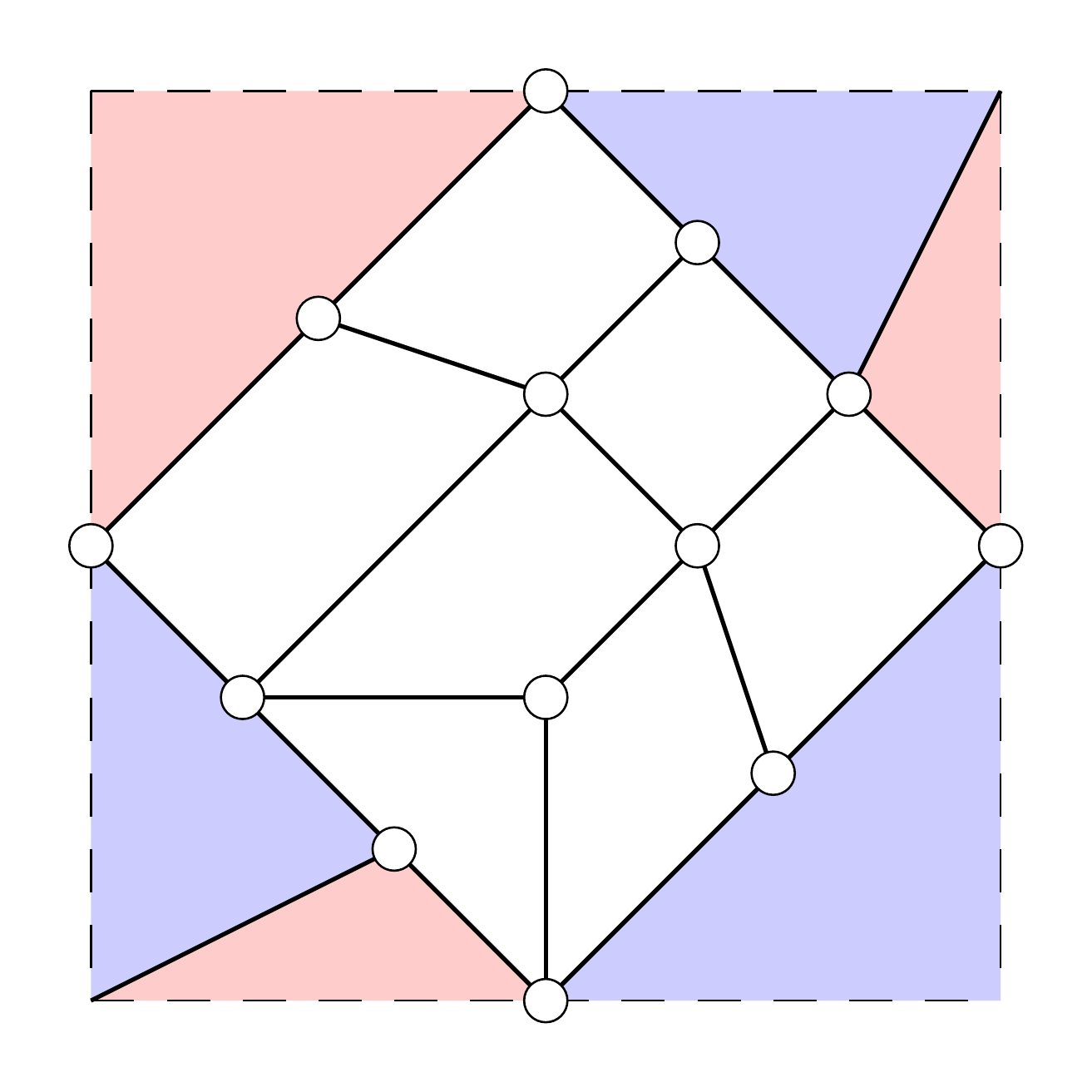}
\caption{$I_7^{b}$}
\end{subfigure}
\caption{Irreducible 4-critical graphs drawn on torus. Red and blue faces are $5-$ and $7-$ faces, respectively. The torus
is obtained by gluing the top edge of the picture with the bottom bottom one and the left one with the right one.}
\label{fig-irreducibles}
\end{figure}

For a set $\TT$ of templates and a graph $G$, $G$ is \emph{represented} by
$\TT$ if there exists $T \in \TT$ such that $G$ is represented by $T$. A set $\TT$ of
templates $\TT$ is \emph{closed under growing} if for every $T \in \TT$,
all templates grown from $T$ belong to $\TT$.
A set $\TT$ of templates is \emph{total} if $\TT$ is a set of relevant critical templates
such that the graphs $I_4$, $I_5$, $I_7^a$, and $I_7^b$ are represented by $\TT$ and $\TT$ is closed under
growing.

\begin{corollary}\label{thm-totalrepr}
If $\TT$ is a total set of templates, then
\begin{itemize}
\item every $4$-critical triangle-free graph drawn in the torus is represented by $\TT$, and
\item no graph represented by $\TT$ is $3$-colorable.
\end{itemize}
\end{corollary}
\begin{proof}
For the first claim, the proof is illustrated in Figure~\ref{fig-example}.
Suppose for a contradiction $G$ is a $4$-critical triangle-free graph drawn in the torus
and not represented by $\TT$ with the smallest number of vertices.  Since all irreducible graphs are represented by $\TT$,
$G$ contains a $4$-face that can be collapsed without creating a triangle; let $H$ be the corresponding reduction of $G$.
Since $H$ is $4$-critical, triangle-free, and $|V(H)|<|V(G)|$, the minimality of $G$ implies $H$ is represented by
a template $T\in \TT$. By Theorem \ref{thm-aplif}, $G$ is represented by a template grown from $T$.
This is a contradiction, since $\TT$ is closed under growing.

The second claim follows from Corollary~\ref{cor-col}, since all templates in $\TT$ are critical, and thus not $3$-colorable.
\end{proof}

Therefore, to finish the proof of Theorem~\ref{thm-main}, it remains to prove the following.

\begin{claim}\label{clm-total}
There exists a total set of direct templates of size $186$.
\end{claim}

This claim is not at all evident, and the fact that we can get away with
only having direct templates came as a bit of a surprise to us. Note that the operations
of filling, boosting, or hiding edges may (and typically do) turn a direct template into a non-direct one. It is fortuitous that adding struts when turning a template into a non-3-colorable
one counteracts these effects.
We do not give a firm theoretical basis to confirm Claim~\ref{clm-total},
but we explicitly construct the set $\TT$ using a computer search.
We give more details in the next section.

\subsection{Generating a total set $\TT$}

In the remainder of this paper, let us describe the process that generates a total set~$\TT$. 
Clearly, it suffices to start with critical templates representing $I_4$, $I_5$, $I_7^a$, and $I_7^b$,
and then close the set under growing. To turn a specific graph $G$ into a template $T$ representing
it, we just set $G_T = G$ and
$$\theta_T(f)=\begin{cases}
\{|f|\}&\text{ if $|f|\ge 5$}\\
\emptyset&\text{ if $|f|=4$}
\end{cases}$$
for each face $f$ of $G$.  We then need a procedure that takes a general
template $T$ that is not $3$-colorable and finds a critical subtemplate. This
can be achieved by iteratively going over all edges of $T$, each time testing
by brute-force whether hiding the edge allows existence of a proper $3$-coloring
of the resulting subtemplate. Once all edges have this property, the
obtained template is critical. Note that although multiple critical
subtemplates may exist, any of them can be selected and so this greedy approach is
sufficient. 

We then perform the following procedure for every template $T$ we add to the set $\TT$:
\begin{itemize}
\item[(1)] Generate all amplifications of $T$,
\item[(2)] for each amplification $T'$ of $T$, find all (up to homeomorphism) $T'$-minimal consistent systems of sets of barriers
that block all proper $3$-colorings of $T'$, and their realizations,
\item[(3)] for each such realization, find a critical subtemplate and add it to $\TT$
if a homeomorphic template is not already present.
\end{itemize}
The step (3) is accomplished using the criticalization procedure from the previous paragraph.
The step (1) is entirely straightforward, following the definition of
partial amplification and amplification and trying all possible choices
(their number generally turns out to be reasonably small, on the order of hundreds).
Let us remark that by Theorem~\ref{thm-centor}, the largest value appearing in the $\theta_T$ sets is $7$,
and thus to perform boosting (and filling), the knowledge of the sets $\SS_k$ for $k\le 9$ given by
Lemma~\ref{lemma-sset} is sufficient.

The step (2) is the most complicated and time-consuming and deserves a more detailed description.
It is achieved by the following recursive procedure, taking as a parameter
a set $\KK$ of $3$-colorings of $T'$ yet to block and a realization $T''$ of a system
of sets of barriers already added to $T'$.

\medskip

$\mathtt{expand\_to\_noncolorable}(\KK,T'')$:
\vspace*{-2mm}
\begin{itemize}
\item If $\KK=\emptyset$, add $\TT''$ to the output and return.
\item Select a coloring $\varphi\in\KK$.
\item For each face $f\in F(G_{T'})$:
\begin{itemize}
\item Let $n_1$, \ldots, $n_k$ be all winding number assignments
for $\theta_{T'}(f)$ summing to $\omega_\varphi(f)$.
\item Generate all sequences $s_1$, \ldots, $s_k$
and $T_0$, $T_1$, \ldots, $T_k$ such that
\begin{itemize}
\item[($\star$)] $T_0=T''$, and for $i=1,\ldots,k$,
\item[($\star\star$)] $s_i=(P_i,I_i)$ is an $f$-barrier for $(\varphi,n_i)$ whose intersection with $G_{T_{i-1}}$ consists
of vertices and edges, and a relevant template $T_i$ is obtained from $T_{i-1}$ by adding $P_i$
to $G_{T_{i-1}}$ and for each face $g\in F(G_{T_{i-1}})$ such that $g\subseteq f$,
redistributing the elements of $\theta_{T_{i-1}}(g)$ among the faces into
which $P_i$ splits $g$ so that exactly the elements belonging to $I_i$ end up in $r(P_i)$.
\end{itemize}
\item For each template $\overline{T}$ appearing as the last
element of such a sequence, let $\overline{\KK}$ be the subset
of $\KK$ consisting of the colorings that extend to $\overline{T}$,
and call $\mathtt{expand\_to\_noncolorable}(\overline{\KK},\overline{T})$.
\end{itemize}
\end{itemize}
The step (2) then amounts to calling $\mathtt{expand\_to\_noncolorable}(\KK,T')$
for the set $\KK$ of all proper $3$-colorings of $T'$ obtained by brute-force.
Let us remark that $\varphi$ does not belong to the set $\overline{\KK}$
by the choice of the $f$-barriers $s_1$, \ldots, $s_k$, and thus the recursion
is guaranteed to terminate.  We do not bother checking that the systems of sets
of barriers obtained in this procedure are $T'$-minimal, as this turned out to
only slow things down---all $T'$-minimal systems are guaranteed to be among those
produced, though, which is sufficient.

Note that the depth of recursion of $\mathtt{expand\_to\_noncolorable}$ may be quite
large (worst case up to the size of $\KK$), and additionally in 
($\star\star$) we need to exhaustively go through a potentially large number of cases.
In practice, the situation is not so bad. Often, adding the $f$-barriers eliminates
many colorings in addition to $\varphi$, limiting the depth of the recursion.
Due to the census limitations given by Theorem~\ref{thm-centor}, the number of
winding number assignments for $\theta_{T'}(f)$ summing to a specific value
is small, we always have $k\le 3$.  Furthermore, once we add a few barriers to a face,
the number of ways to select a barrier in ($\star\star$) satisfying the intersection and realization
conditions is limited.  Minor further optimizations (choosing the coloring $\varphi$
so that as few choices in ($\star\star$) are possible over all faces $f$, caching to ensure
we do not process the same template more than once) are sufficient to ensure
that the whole enumeration process finishes within few hours on standard hardware.

Each of the authors wrote an implementation of the described procedure, to decrease
the chance of programming errors.  Both implementations can be found at \href{https://iuuk.mff.cuni.cz/~rakdver/torus/}{\url{https://iuuk.mff.cuni.cz/~rakdver/torus/}}.

\bibliographystyle{acm}
\bibliography{bibliography}

\section*{Appendix}

Here we list the templates from the total set $\TT$.  The drawing on torus is obtained
by gluing the left and the right side to form a cylinder, then gluing the two boundary
cycles of the cylinder, as indicated by the vertex labels.  The grey faces are quadrangulated,
while the white, red, green, and blue faces have lengths four, five, six, and seven,
respectively.

\includegraphics[page=1,width=\textwidth]{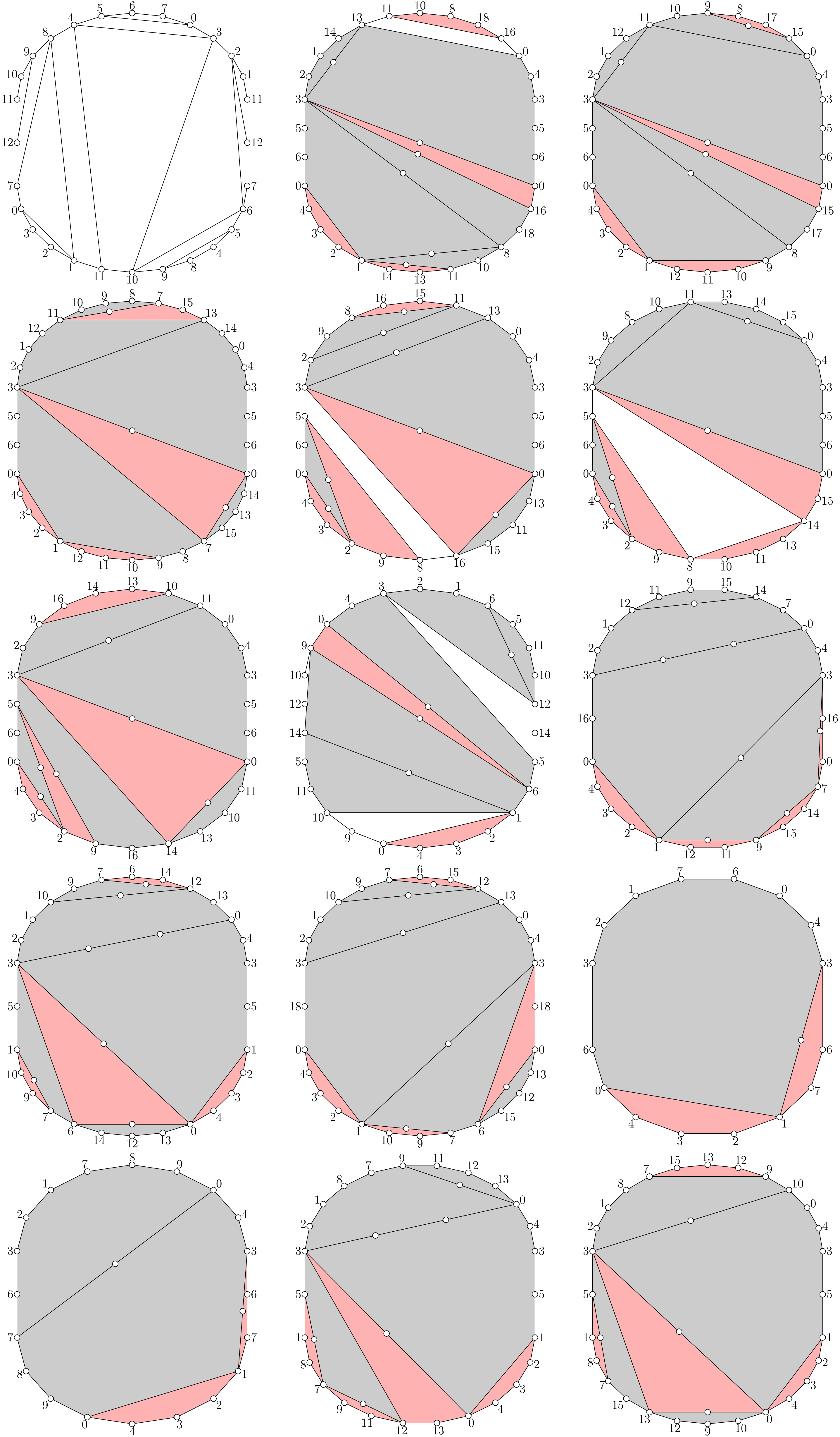}

\includegraphics[page=2,width=\textwidth]{fig-critical.pdf}

\includegraphics[page=3,width=\textwidth]{fig-critical.pdf}

\includegraphics[page=4,width=\textwidth]{fig-critical.pdf}

\includegraphics[page=5,width=\textwidth]{fig-critical.pdf}

\includegraphics[page=6,width=\textwidth]{fig-critical.pdf}

\includegraphics[page=7,width=\textwidth]{fig-critical.pdf}

\includegraphics[page=8,width=\textwidth]{fig-critical.pdf}

\includegraphics[page=9,width=\textwidth]{fig-critical.pdf}

\includegraphics[page=10,width=\textwidth]{fig-critical.pdf}

\includegraphics[page=11,width=\textwidth]{fig-critical.pdf}

\includegraphics[page=12,width=\textwidth]{fig-critical.pdf}

\includegraphics[page=13,width=\textwidth]{fig-critical.pdf}

\end{document}